\documentclass[10pt,oneside,english]{amsart}
\usepackage[T1]{fontenc}
\usepackage[latin9]{inputenc}
\usepackage[a4paper]{geometry}
\usepackage{amsthm}
\usepackage{amstext}
\usepackage{amssymb}
\usepackage{xypic}
\usepackage{pstricks,pst-node,pst-tree}
\usepackage{babel}

\makeatletter
\numberwithin{equation}{section}
\numberwithin{figure}{section}
  \theoremstyle{plain}
\newtheorem*{thm*}{Theorem}
\theoremstyle{plain}
\newtheorem{thm}{Theorem}[section]
 \theoremstyle{definition}
\newtheorem{example}[thm]{Example}
\theoremstyle{plain}
\newtheorem{prop}[thm]{Proposition}
\theoremstyle{remark}
\newtheorem{rem}[thm]{Remark}
\theoremstyle{definition}
\newtheorem{defn}[thm]{Definition}
\theoremstyle{plain}
\newtheorem{lem}[thm]{Lemma}
\theoremstyle{plain}
\newtheorem{cor}[thm]{Corollary}
\theoremstyle{definition}
\newtheorem{parn}{}[section]
\makeatother

\begin{document}
\title{On exotic affine $3$-spheres}

\author{Adrien Dubouloz} 
\address{Adrien Dubouloz, Institut de Math\'ematiques de Bourgogne, Universit\'e de Bourgogne, 9 avenue Alain Savary - BP 47870, 21078 Dijon cedex, France} 
\email{Adrien.Dubouloz@u-bourgogne.fr}

\author{David R. Finston} 
\address{Department of Mathematical Sciences, New Mexico
State University, Las Cruces, New Mexico 88003} 
\email{dfinston@nmsu.edu}

\begin{abstract}
Every $\mathbb{A}^{1}-$bundle over $\mathbb{A}_{\ast }^{2},$ the complex
affine plane punctured at the origin, is trivial in the differentiable category
but there are infinitely many distinct isomorphy classes of algebraic
bundles. Isomorphy types of total spaces of such algebraic bundles are
considered; in particular, the complex affine $3$-sphere $\mathbb{S}_{%
\mathbb{C}}^{3},$ given by $z_{1}^{2}+z_{2}^{2}+z_{3}^{2}+z_{4}^{2}=1,$
admits such a structure with an additional homogeneity property. Total
spaces of nontrivial homogeneous $\mathbb{A}^{1}$-bundles over $\mathbb{A}%
_{\ast }^{2}$ are classified up to $\mathbb{G}_{m}$-equivariant algebraic
isomorphism  and a criterion for nonisomorphy is given. 
In fact $\mathbb{S}_{\mathbb{C}}^{3}$ is not isomorphic as an
abstract variety to the total space of any $\mathbb{A}^{1}$-bundle over 
$\mathbb{A}_{\ast }^{2}$ of different homogeneous degree, which gives rise to the existence
of exotic spheres, a phenomenon that first arises in dimension three. 
As a by product, an example is given of two biholomorphic but not algebraically 
isomorphic threefolds, both with a trivial Makar-Limanov invariant, and with isomorphic
cylinders.
\end{abstract}


\maketitle

\section*{Introduction}

Exotic affine spaces emerged in the 1990's as rather unusual objects in
affine algebraic geometry. These are smooth complex affine varieties
diffeomorphic to a euclidean space but not algebraically isomorphic to the
usual affine space. Actually, the first examples were constructed by
Ramanujam in a landmark paper \cite{Ramanujam1971} in which he also
established the non existence of exotic affine planes. Since then, many
other examples of smooth contractible affine varieties of any dimension $%
n\geq 3$ have been discovered and these objects have progressively become
ubiquitous in affine algebraic geometry. The study of these potential exotic 
$\mathbb{A}^{n}$'s has been a motivation for the introduction and the
development of new techniques and "designer" invariants which in turn led to
important progress in related questions, such as the Zariski Cancellation
Problem (see e.g. \cite{Zaidenberg1999} for a survey). So far, these
invariants have succeeded in distinguishing certain of these varieties from
usual affine spaces, most notably the famous Russell cubic threefold $%
X=\left\{ x^{2}y+z^{2}+x+t^{3}=0\right\} \subset \mathbb{A}^{4}$ \cite%
{Kaliman1999,Makar-Limanov1996}. But the main difficulty still remains the
lack of effective tools to recognize exotic spaces or, equivalently, the
lack of effective characterizations of affine spaces among affine varieties.

More generally, given any smooth complex affine variety $V$, one can ask if
there exists smooth affine varieties $W$ non isomorphic to $V$ but which are
biholomorphic or diffeomorphic to $V$ when equipped with their underlying
structures of complex analytic or differentiable manifolds. When such exist,
these varieties $W$ could be called exotic algebraic structures on $V,$ but
it makes more sense to reserve this terminology for the case where the chosen
variety $V$ carries an algebraic structure that we consider as the
\textquotedblleft usual\textquotedblright\ one.

In addition to affine spaces, a very natural class for which we have such
\textquotedblleft usual\textquotedblright\ algebraic structures consists of
non-degenerate smooth complex affine quadrics, i.e., varieties isomorphic to
one of the form $\mathbb{S}_{\mathbb{C}}^{n}=\left\{ x_{1}^{2}+\cdots
+x_{n+1}^{2}=1\right\} $ equipped with its unique structure of a closed
algebraic subvariety of $\mathbb{A}_{\mathbb{C}}^{n+1}$. So an\emph{\ exotic
complex affine $n$-sphere, }if it exists, will be a smooth complex affine
variety diffeomorphic to $\mathbb{S}_{\mathbb{C}}^{n}$ but not algebraically
isomorphic to it. Since $\mathbb{S}_{\mathbb{C}}^{1}\simeq \mathbb{A}%
^{1}\setminus \left\{ 0\right\} $ is the unique smooth affine curve $C$ with 
$H_{1}\left( C,\mathbb{Z}\right) \simeq \mathbb{Z}$, there is no exotic
affine $1$-sphere. Similarly, there is no exotic affine $2$-sphere and the
same phenomenon as for affine spaces occurs: the algebraic structure on a
smooth affine surface diffeomorphic to $\mathbb{S}_{\mathbb{C}}^{2}$ is
actually uniquely determined by its topology: namely, a smooth affine
surface $S$ is algebraically isomorphic to $\mathbb{S}_{\mathbb{C}}^{2}$ if
and only if it has the same homology type and the same homotopy type at
infinity as $\mathbb{S}_{\mathbb{C}}^{2}$ (see \ref{thm:NoExotic2-sphere} in
the appendix below).

In the context of the cancellation problem for factorial threefolds, S.
Maubach and the second author \cite{Finston2008} studied a family of smooth
affine threefolds with the homology type of $\mathbb{S}_{\mathbb{C}}^{3}$ :
starting from a Brieskorn surface $S_{p,q,r}=\left\{
x^{p}+y^{q}+z^{r}=0\right\} \subset \mathbb{A}^{3}$, they consider smooth
affine threefolds $Z_{m,n}\subset S_{p,q,r}\times \mathbb{A}^{2}$ defined by
equations of the form $x^{m}v-y^{n}u=1$, $m\geq n\geq 1$. These varieties
come equipped via the first projection with the structure of a locally
trivial $\mathbb{A}^{1}$-bundle $\rho :Z_{m,n}\rightarrow S_{p,q,r}^{\ast }$
over the smooth locus $S_{p,q,r}^{\ast }=S_{p,q,r}\setminus \{(0,0,0)\}$ of $%
S_{p,q,r}$ . For a fixed triple $\left( p,q,r\right) $, they are all
diffeomorphic to each other and have the Brieskorn sphere $\Sigma \left(
p,q,r\right) $ as a strong deformation retract. The main result of \cite%
{Finston2008} asserts in contrast that if $\frac{1}{p}+\frac{1}{q}+\frac{1}{r%
}<1$, then the isomorphy type of the total space of an $\mathbb{A}^{1}$%
-bundle over $S_{p,q,r}^{\ast }$ as an abstract algebraic variety is
uniquely determined by its isomorphy class as an $\mathbb{A}^{1}$-bundle
over $S_{p,q,r}^{\ast }$, up to composition by automorphisms of $%
S_{p,q,r}^{\ast }$\footnote{%
In loc. cit., this property is established by algebraic methods involving
the computation of the Makar-Limanov invariant of the varieties $Z_{m,n}$,
but this can also be seen alternatively as consequence of the fact for $%
1/p+1/q+1/r<1$, the logarithmic Kodaira dimension $\overline{\kappa }%
(S_{p,q,r}^{*})$ of $S_{p,q,r}^{*}$ is positive.}. This enables in
particular the conclusion that the $Z_{m,n}$ are pairwise non isomorphic
algebraic varieties, despite the isomorphy of all of the cylinders $
Z_{m,n}\times \mathbb{A}^{1}$.

Noting that $\mathbb{S}_{\mathbb{C}}^{3}\cong \mathrm{SL}_{2}\left( \mathbb{C%
}\right) =\left\{ xv-yu=1\right\} \subset \mathbb{A}_{\ast }^{2}\times 
\mathbb{A}^{2}$, where $\mathbb{A}_{\ast }^{2}=\mathrm{Spec}\left( \mathbb{C}%
\left[ x,y\right] \right) \setminus \{(0,0)\}$, the previous result strongly
suggests that the varieties 
\begin{equation*}
X_{m,n}=\left\{ x^{m}v-y^{n}u=1\right\} \subset \mathbb{A}_{\ast }^{2}\times 
\mathbb{A}^{2},\quad m+n>2,
\end{equation*}%
could be exotic affine $3$-spheres. Indeed, for every $m,n\geq 1$, the first
projection again induces a Zariski locally trivial $\mathbb{A}^{1}$-bundle $%
\rho :X_{m,n}\rightarrow \mathbb{A}_{\ast }^{2}$. The latter being trivial
in the euclidean topology, the $X_{m,n}$ are thus all diffeomorphic to $%
\mathbb{A}_{\ast }^{2}\times \mathbb{R}^{2}$ and have the real sphere $S^{3}$
as a strong deformation retract. This holds more generally for any Zariski
locally trivial $\mathbb{A}^{1}$-bundle $\rho :X\rightarrow \mathbb{A}_{\ast
}^{2}$, and so all smooth affine threefolds admitting such a structure are
natural candidates for being exotic $3$-spheres. But it turns out that it is
a challenging problem to distinguish these varieties from $\mathbb{S}_{%
\mathbb{C}}^{3}\simeq X_{1,1}$ since, in contrast with the situation
considered in \cite{Finston2008}, the isomorphy type of the total space of
an $\mathbb{A}^{1}$-bundle $\rho :X\rightarrow \mathbb{A}_{\ast }^{2}$ as an
abstract variety is no longer uniquely determined by its structure as an $%
\mathbb{A}^{1}$-bundle; for instance, a consequence of our main result is
the following rather unexpected fact that for every pair $\left( m,n\right)
,\left( p,q\right) \in \mathbb{Z}_{>0}^{2}$ such that $m+n=p+q\geq 4$, the
threefolds $X_{m,n}$ and $X_{p,q}$ are isomorphic as abstract algebraic
varieties while they are isomorphic as $\mathbb{A}^{1}$-bundles over $%
\mathbb{A}_{\ast }^{2}$ only if $\left\{ m,n\right\} =\left\{ p,q\right\} $.%
\footnote{%
This fact was actually already observed by the authors and P.D Metha in the
unpublished paper  \cite{DubFinMet2009}.}

While a complete and effective classification of isomorphy types of total
spaces of $\mathbb{A}^{1}$- bundles over $\mathbb{A}_{\ast }^{2}$ seems out
of reach for the moment, we obtain a satisfactory answer for a particular
class of bundles containing the varieties $X_{m,n}$ that we call \emph{%
homogeneous} $\mathbb{G}_{a}$-\emph{bundles}. These are principal $\mathbb{G}%
_{a}$-bundles $\rho :X\rightarrow \mathbb{A}_{\ast }^{2}$ equipped with a
lift of the $\mathbb{G}_{m}$-action $\lambda \cdot \left( x,y\right) =\left(
\lambda x,\lambda y\right) $ on $\mathbb{A}_{\ast }^{2}$ which is
\textquotedblleft locally linear\textquotedblright\ on the fibers of $\rho $%
. This holds for instance on the $X_{m,n}$'s for the lifts $\lambda \cdot
\left( x,y,u,v\right) =\left( \lambda x,\lambda y,\lambda ^{-n}u,\lambda
^{-m}v\right) $, which are the analogues of the action of the maximal torus
of $X_{1,1}=\mathrm{SL}_{2}\left( \mathbb{C}\right) $ on $\mathrm{SL}%
_{2}\left( \mathbb{C}\right) $ by multiplication on the right. For such
bundles, there is natural notion of homogeneous degree for which, in
particular, the bundle $X_{m,n}\rightarrow \mathbb{A}_{\ast }^{2}$ equipped
with the previous lift has homogeneous degree $-m-n$. Our main
classification result then reads as follows (see Theorem \ref%
{thm:HomBunIsoType}):

\begin{thm*}
The total spaces of two nontrivial homogeneous $\mathbb{G}_{a}$-bundles are $%
\mathbb{G}_{m}$-equivariantly isomorphic if and only if they have the same
homogeneous degree. In particular, for a fixed $d\geq2,$ the total spaces of
nontrivial homogeneous $\mathbb{G}_{a}$-bundles $\rho:X\rightarrow\mathbb{A}%
_{*}^{2}$ of degree $-d$ are all isomorphic as abstract affine varieties.
\end{thm*}

This implies in particular that a variety $X_{m,n}$ with $m+n\geq 3$
equipped with the action above is not $\mathbb{G}_{m}$-equivariantly
isomorphic to $X_{1,1}$. We finally derive from a careful study of the
effect of algebraic isomorphisms on the algebraic de Rham cohomology of
total spaces of $\mathbb{A}^{1}$-bundles over $\mathbb{A}_{\ast }^{2}$ that
every variety $X_{m,n}$ with $m+n\geq 3$ is indeed an exotic affine $3$%
-sphere. The criterion we give in Theorem \ref{thm:DeRhamArg} also provides
an effective tool to construct families of pairwise non isomorphic exotic $3$%
-spheres : for instance, we show that the varieties $X_{2,2}=\left\{
x^{2}v-y^{2}u=1\right\} $ and $\tilde{X}_{2,2}=\left\{
x^{2}v-y^{2}u=1+xy\right\} $ are non isomorphic exotic affine $3$-spheres
yet they are even biholomorphic as complex analytic manifolds. \newline

The article is organized as follows. In the first section, the basic
properties of $\mathbb{A}^{1}$-bundles over the punctured plane $\mathbb{A}%
_{\ast }^{2}$ are reviewed and the notion of homogeneous $\mathbb{G}_{a}$%
-bundle is developed. The second section is devoted to the proofs of the
various isomorphy criteria presented above. The third section takes the form
of an appendix, in which we give a short proof of the non existence of
exotic affine $2$-spheres and establish a refined version of the so-called
Danilov-Gizatullin Isomorphy Theorem \cite{Gizatullin1977} which is used in
the proof of the main Theorem \ref{thm:HomBunIsoType}

\section{Basic facts on $\mathbb{A}^{1}$-bundles and $\mathbb{G}_{a}$%
-bundles over $\mathbb{A}_{*}^{2}$}

\subsection{Recollection on algebraic $\mathbb{A}^{1}$-bundles and $\mathbb{G%
}_{a}$-bundles}

\indent\newline
\noindent An $\mathbb{A}^{1}$-bundle over a scheme $S$ is a morphism $\rho
:X\rightarrow S$ for which every point of $S$ has a Zariski open
neighborhood $U\subset S$ with a local trivialization such that $\rho
^{-1}\left( U\right) \simeq U\times \mathbb{A}^{1}$ as schemes over $U$.
Transition isomorphisms over the intersections $U_{i}\cap U_{j}$ of pairs of
such open sets are given by affine transformations of the fiber isomorphy classes 
of such bundles are thus in one-to-one
correspondence with isomorphy classes of Zariski locally trivial principal
bundles under the affine group $\mathrm{Aut}\left( \mathbb{A}^{1}\right)
\simeq \mathbb{G}_{m}\ltimes \mathbb{G}_{a}$. Additional properties of these
bundles can be read from the exact sequence of non-abelian cohomology

\begin{equation*}
0\rightarrow H^{0}\left( S,\mathbb{G}_{a}\right) \rightarrow H^{0}\left( S,%
\mathbb{G}_{m}\ltimes \mathbb{G}_{a}\right) \rightarrow H^{0}\left( S,%
\mathbb{G}_{m}\right) \rightarrow H^{1}\left( S,\mathbb{G}_{a}\right)
\rightarrow H^{1}\left( S,\mathbb{G}_{m}\ltimes \mathbb{G}_{a}\right)
\rightarrow H^{1}\left( S,\mathbb{G}_{m}\right)
\end{equation*}%
deduced from the short exact sequence of groups $0\rightarrow \mathbb{G}%
_{a}\rightarrow \mathbb{G}_{m}\ltimes \mathbb{G}_{a}\rightarrow \mathbb{G}%
_{m}\rightarrow 0$ (see e.g. \cite[3.3.1]{Giraud1971}). For instance, if $S$
is affine then $H^{1}(S,\mathbb{G}_{a})\simeq H^{1}(S,\mathcal{O}%
_{S})=\left\{ 0\right\} $ and so every $\mathbb{A}^{1}$-bundle over $S$
actually carries the structure of a line bundle. Similarly, if the Picard
group $\mathrm{Pic}\left( S\right) \simeq H^{1}(S,\mathbb{G}_{m})$ of $S$ is
trivial then every $\mathbb{A}^{1}$-bundle over $S$ can be equipped with the
additional structure of a principal $\mathbb{G}_{a}$-bundle. Furthermore, in
this case, the set $H^{1}(S,\mathbb{G}_{m}\ltimes \mathbb{G}_{a})$ of
isomorphy classes of $\mathbb{A}^{1}$-bundles over $S$ is isomorphic to the
quotient of $H^{1}(S,\mathbb{G}_{a})$ by the action of $H^{0}(S,\mathbb{G}%
_{m})=\Gamma (S,\mathcal{O}_{S}^{\ast })$ via a multiplicative
reparametrization: $a\in \Gamma (S,\mathcal{O}_{S}^{\ast })$ sends the
isomorphy class of the $\mathbb{G}_{a}$-bundle $\rho :X\rightarrow S$ with
action $\mathbb{G}_{a,S}\times _{S}X\rightarrow X$, $\left( t,x\right)
\mapsto t\cdot x$ to the isomorphy class of $\rho :X\rightarrow S$ equipped
with the action $\left( t,x\right) \mapsto \left( at\right) \cdot x$. 

\begin{parn}
It follows in particular that over the punctured affine plane $\mathbb{A}%
_{\ast }^{2}$, which has a trivial Picard group, the notions of $\mathbb{A}%
^{1}$-bundle and $\mathbb{G}_{a}$-bundle essentially coincide and that
isomorphy classes of nontrivial $\mathbb{A}^{1}$-bundles are in one-to-one
correspondence with elements of the infinite dimensional projective space
$\mathbb{P}H^{1}(\mathbb{A}_{\ast }^{2},\mathcal{O}_{\mathbb{A}_{\ast }^{2}})=H^{1}(\mathbb{A}_{\ast }^{2},\mathcal{O}_{\mathbb{A}_{\ast }^{2}})/\mathbb{G}_m$.
\end{parn}

\begin{parn}
Every cohomology class in $H^{1}(\mathbb{A}_{*}^{2},\mathcal{O}_{\mathbb{A}%
_{*}^{2}})$ can be represented by a \v{C}ech $1$-cocycle with value in $%
\mathcal{O}_{\mathbb{A}_{*}^{2}}$ on the the acyclic covering $\mathcal{U}%
_{0}$ of $\mathbb{A}_{*}^{2}=\mathrm{Spec}\left(\mathbb{C}\left[x,y\right]%
\right)\setminus\left\{ 0,0\right\} $ by the principal affine open subsets $%
U_{x}=\mathrm{Spec}\left(\mathbb{C}\left[x^{\pm1},y\right]\right)$ and $%
U_{y}=\mathrm{Spec}\left(\mathbb{C}\left[x,y^{\pm1}\right]\right)$,
providing an isomorphism of $\mathbb{C}$-vector spaces 
\begin{equation*}
H^{1}(\mathbb{A}_{*}^{2},\mathcal{O}_{\mathbb{A}_{*}^{2}})\simeq\check{H}%
^{1}(\mathcal{U}_{0},\mathcal{O}_{\mathbb{A}_{*}^{2}})\simeq\mathbb{C}\left[%
x^{\pm1},y^{\pm1}\right]/\langle\mathbb{C}\left[x^{\pm1},y\right]+\mathbb{C}%
\left[x,y^{\pm1}\right]\rangle.
\end{equation*}
It follows in particular from this description that $H^{1}(\mathbb{A}%
_{*}^{2},\mathcal{O}_{\mathbb{A}_{*}^{2}})$ is nonzero, which implies in
turn that there exists nontrivial algebraic $\mathbb{G}_{a}$-bundles over $%
\mathbb{A}_{*}^{2}$. In contrast, in the differentiable category, all these
bundles are globally trivial smooth fibrations with fibers $\mathbb{R}^{2}$
over $\mathbb{A}_{*}^{2}\simeq\mathbb{R}^{4}\setminus\left\{ 0\right\} $
equipped with its euclidean structure of differentiable manifold. Indeed,
since the sheaf $\mathcal{F}=\mathcal{C}^{\infty}(\mathbb{A}_{*}^{2},\mathbb{%
C})$ of complex valued $\mathcal{C}^{\infty}$-functions on $\mathbb{A}%
_{*}^{2}$ is soft (see e.g. \cite[Theorem 5, p.25]{Grauert1979}), every
algebraic \v{C}ech $1$-cocycle $g\in C^{1}(\mathcal{U}_{0},\mathcal{O}_{%
\mathbb{A}_{*}^{2}})$ representing a nontrivial class in $H^{1}(\mathbb{A}%
_{*}^{2},\mathcal{O}_{\mathbb{A}_{*}^{2}})$ is a coboundary when considered
as a $1$-cocycle in with values in $\mathcal{F}$. This implies in particular
that every algebraic $\mathbb{G}_{a}$-bundle $\rho:X\rightarrow\mathbb{A}%
_{*}^{2}$ admits the real sphere $S^3$ as a $%
\mathcal{C}^{\infty}$-strong deformation retract.
\end{parn}

\begin{example}
A well known example of nontrivial algebraic $\mathbb{G}_{a}$-bundle over $%
\mathbb{A}_{\ast }^{2}$ is given by the morphism 
\begin{equation*}
\rho :\mathrm{SL}_{2}\left( \mathbb{C}\right) =\left\{ \left( 
\begin{array}{cc}
x & u \\ 
y & v%
\end{array}%
\right) \in \mathcal{M}_{2}\left( \mathbb{C}\right) ,\,xv-yu=1\right\}
\rightarrow \mathbb{A}_{\ast }^{2},\left( 
\begin{array}{cc}
x & u \\ 
y & v%
\end{array}%
\right) \mapsto \left( x,y\right)
\end{equation*}%
which identifies $\mathbb{A}_{\ast }^{2}$ with the quotient of $\mathrm{SL}%
_{2}\left( \mathbb{C}\right) $ by the right action of its subgroup $T\simeq 
\mathbb{G}_{a}$ of upper triangular matrices with $1$'s on the diagonal. The
local trivializations of $\rho $ given by the $\mathbb{G}_{a}$-equivariant
isomorphisms $\mathrm{SL}_{2}\left( \mathbb{C}\right) \mid _{U_{x}}\simeq
U_{x}\times \mathrm{Spec}\left( \mathbb{C}\left[ x^{-1}u\right] \right) $
and $\mathrm{SL}_{2}\left( \mathbb{C}\right) \mid _{U_{y}}\simeq U_{y}\times 
\mathrm{Spec}\left( \mathbb{C}\left[ y^{-1}v\right] \right) $ differ over $%
U_{x}\cap U_{y}$ by the nontrivial \v{C}ech $1$-cocycle $(xy)^{-1}\in C^{1}(%
\mathcal{U}_{0},\mathcal{O}_{\mathbb{A}_{\ast }^{2}})=\mathbb{C}\left[
x^{\pm 1},y^{\pm 1}\right] .$ In contrast, the identity $(xy)^{-1}=\delta
(x,y){}^{-1}(\overline{x}y^{-1}+x^{-1}\overline{y})$, where $\delta \left(
x,y\right) =\left\vert x\right\vert ^{2}+\left\vert y\right\vert ^{2}\in 
\mathcal{C}^{\infty }(\mathbb{A}_{\ast }^{2},\mathbb{R}_{+}^{\ast })$, shows
that $(xy)^{-1}\in C^{1}(\mathcal{U}_{0},\mathcal{C}^{\infty }(\mathbb{A}%
_{\ast }^{2},\mathbb{C}))$ is a coboundary.
\end{example}

\subsection{Algebraic $\mathbb{G}_{a}$-bundles with affine total spaces}

\indent\newline
\noindent Since $\mathbb{A}_{\ast }^{2}$ is strictly quasi-affine, the total
space of a $\mathbb{G}_{a}$-bundle over it need not be an affine variety in
general. However, we have the following handsome characterization of $%
\mathbb{G}_{a}$-bundles with affine total spaces.

\begin{prop}
\label{pro:AffineCrit} A $\mathbb{G}_{a}$-bundle $\rho:X\rightarrow\mathbb{A}%
_{*}^{2}$ has affine total space $X$ if and only if it is nontrivial.
\end{prop}

\begin{proof}
See also \cite{DubFinMet2009} for an alternative argument.
The condition is necessary since the total space of the trivial $\mathbb{A}%
^{1}$-bundle $\mathbb{A}_{\ast }^{2}\times \mathbb{A}^{1}$ is again strictly
quasi-affine. Conversely, it is equivalent to show that if $\rho
:X\rightarrow \mathbb{A}_{\ast }^{2}$ is a nontrivial $\mathbb{G}_{a}$%
-bundle then the composition 
\begin{equation*}
i\circ \rho :X\rightarrow \mathbb{A}_{\ast }^{2}\hookrightarrow \mathbb{A}%
^{2}=\mathrm{Spec}(\Gamma (\mathbb{A}_{\ast }^{2},\mathcal{O}_{\mathbb{A}%
_{\ast }^{2}}))
\end{equation*}%
is an affine morphism. Let $o=\left( 0,0\right) $ be the origin of $\mathbb{A%
}^{2}$ and $S=\mathrm{Spec}(\mathcal{O}_{\mathbb{A}^{2},o})\rightarrow 
\mathbb{A}^{2}$ be the natural morphism. Since $\rho :X\rightarrow \mathbb{A}%
_{\ast }^{2}$ is an affine morphism, $i\circ \rho $ is affine if and only if
the base extension $\mathrm{p}_{1}:S\times _{\mathbb{A}^{2}}X\rightarrow S$
is as well. It follows from a characterization due to Miyanishi \cite%
{Miyanishi1988} that either the restriction of $\mathrm{p}_{1}$ over the
complement of the closed point $o$ of $S$ is a trivial $\mathbb{A}^{1}$%
-bundle or $S\times _{\mathbb{A}^{2}}X$ is isomorphic to a closed subscheme
of $S\times \mathrm{Spec}\left( \mathbb{C}\left[ u,v\right] \right) $
defined by an equation of the form $fv-gu=1$, where $f,g\in \mathcal{O}_{%
\mathbb{A}^{2},o}$ is a regular sequence. In the second case, $\mathrm{p}%
_{1}^{-1}\left( S\right) =\mathrm{p}_{1}^{-1}(S\setminus \left\{ o\right\} )$
is an affine scheme, and so $\mathrm{p}_{1}$ is an affine morphism.
Otherwise, if $\mathrm{p}_{1}^{-1}(S\setminus \left\{ o\right\} )\rightarrow
S\setminus \left\{ o\right\} $ is the trivial $\mathbb{A}^{1}$-bundle, then
there exists an open subset $U$ of $\mathbb{A}^{2}$ containing $o$ such that 
$\rho ^{-1}(U\setminus \left\{ o\right\} )\rightarrow U\setminus \left\{
o\right\} $ is a trivial $\mathbb{G}_{a}$-bundle. Therefore, $\rho
:X\rightarrow \mathbb{A}_{\ast }^{2}$ can be extended to a $\mathbb{G}_{a}$%
-bundle $\overline{\rho }:\overline{X}\rightarrow \mathbb{A}^{2}$ over $%
\mathbb{A}^{2}$. But $\mathbb{A}^{2}$ is affine, so $\overline{\rho }:%
\overline{X}\rightarrow \mathbb{A}^{2}$ is a trivial $\mathbb{G}_{a}$%
-bundle, and as would be $\rho :X\rightarrow \mathbb{A}_{\ast }^{2}$,
contradicting our hypothesis.
\end{proof}

\begin{example}
\label{exa:LocalDesc} Every non trivial class in $H^{1}(\mathbb{A}_{\ast
}^{2},\mathcal{O}_{\mathbb{A}_{\ast }^{2}})$ can be represented by a \v{C}%
ech $1$-cocycle of the form $x^{-m}y^{-n}p\left( x,y\right) \in C^{1}(%
\mathcal{U}_{0},\mathcal{O}_{\mathbb{A}_{\ast }^{2}})=\mathbb{C}\left[
x^{\pm 1},y^{\pm 1}\right] $, where $m,n\in \mathbb{N}\setminus \left\{
0\right\} $ and $p\left( x,y\right) \in \mathbb{C}\left[ x,y\right] $ is a
polynomial divisible neither by $x$ nor by $y$ and satisfying $\deg _{x}p<m$
and $\deg _{y}p<n$. A corresponding $\mathbb{A}^{1}$-bundle $\rho :X\left(
m,n,p\right) \rightarrow \mathbb{A}_{\ast }^{2}$ is obtained as the
complement in the variety 
\begin{equation*}
Z_{m,n,p}=\left\{ x^{m}v-y^{n}u=p\left( x,y\right) \right\} \subset \mathbb{A%
}_{\ast }^{2}\times \mathrm{Spec}\left( \mathbb{C}\left[ u,v\right] \right)
\end{equation*}%
of the fiber $\mathrm{pr}_{x,y}\mid _{Z_{m,n,p}}^{-1}\left( 0,0\right) $.
The latter is a $\mathbb{G}_{a}$-bundle when equipped for instance with the
restriction of the $\mathbb{G}_{a}$-action $t\cdot \left( x,y,u,v\right)
=\left( x,y,u+x^{m}t,v+y^{n}v\right) $ on $Z_{m,n,p}$. Indeed, by
construction of $X\left( m,n,p\right) $, with respect to the local $\mathbb{G%
}_{a}$-equivariant trivializations 
\begin{equation*}
X\left( m,n,p\right) \mid _{U_{x}}\simeq U_{x}\times \mathrm{Spec}\left( 
\mathbb{C}\left[ x^{-m}u\right] \right) \simeq U_{x}\times \mathbb{G}%
_{a}\quad \text{and}\quad X\left( m,n,p\right) \mid _{U_{y}}\simeq
U_{y}\times \mathrm{Spec}\left( \mathbb{C}\left[ y^{-n}v\right] \right)
\simeq U_{y}\times \mathbb{G}_{a}
\end{equation*}
the fiber coordinates differ precisely by the \v{C}ech $1$-cocycle $%
x^{-m}y^{-n}p\left( x,y\right) \in C^{1}(\mathcal{U}_{0},\mathcal{O}_{%
\mathbb{A}_{\ast }^{2}})$.
\end{example}

\begin{parn}
\label{par:cancelRq} A consequence of Proposition \ref{pro:AffineCrit} above
is that the cylinders $X\times \mathbb{A}^{1}$ over the total spaces of
nontrivial $\mathbb{G}_{a}$-bundles $\rho :X\rightarrow \mathbb{A}_{\ast
}^{2}$ are not only all diffeomorphic in the euclidean topology but even all
isomorphic as algebraic varieties. Indeed, given to such bundles $X$ and $%
X^{\prime }$, the fiber product $X\times _{\mathbb{A}_{\ast }^{2}}X^{\prime
} $ is a $\mathbb{G}_{a}$-bundle over both $X$ and $X^{\prime }$ via the
first and the second projections respectively. Since $X$ and $X^{\prime }$
are affine, the latter are both trivial as bundles over $X$ and $X^{\prime }$ 
respectively providing isomorphisms $X\times \mathbb{A}%
^{1}\simeq X\times _{\mathbb{A}_{\ast }^{2}}X^{\prime }\simeq X^{\prime
}\times \mathbb{A}^{1}$. This implies in turn that most of the
\textquotedblleft standard\textquotedblright\ invariants of algebraic
varieties which are either of topologico-differentiable nature or stable
under taking cylinders fail to distinguish total spaces of nontrivial $%
\mathbb{G}_{a}$-bundles $X\rightarrow \mathbb{A}_{\ast }^{2}$ considered as
abstract affine varieties. For instance, algebraic vector bundles on the
total space of a nontrivial $\mathbb{A}^{1}$-bundle $\rho :X\rightarrow 
\mathbb{A}_{\ast }^{2} $ are all trivial. Indeed, this holds for $\mathrm{SL}%
_{2}\left( \mathbb{C}\right) $ by virtue of \cite{Murthy2000} and the fact
that vector bundles on $X\times \mathbb{A}^{1}\simeq \mathrm{SL}_{2}\left( 
\mathbb{C}\right) \times \mathbb{A}^{1}$ are simultaneously extended from
vector bundles on $\mathrm{SL}_{2}\left( \mathbb{C}\right) $ and $X$ \cite%
{Lindel1981/82} implies that this holds for arbitrary affine $X$ too.
\end{parn}

\begin{rem}
\label{rem:ML-triv} In the context of affine varieties $V$ with $\mathbb{G}%
_{a}$-actions, the \emph{Makar-Limanov invariant} $\mathrm{ML}\left(
V\right) $, defined as the algebra of regular functions on $V$ that are
invariant under \emph{all} algebraic $\mathbb{G}_{a}$-actions on $V$, has
been introduced and used by Makar-Limanov \cite{Makar-Limanov1996} to
distinguish certain exotic algebraic structures on the affine $3$-space.
Clearly, $\mathrm{ML}\left( \mathrm{SL}_{2}\left( \mathbb{C}\right) \right) =%
\mathbb{C}$, and since this invariant is known to be unstable under taking
cylinders \cite{BandmanML2005} \cite{DuboulozTG2009}, it is a natural
candidate for distinguishing certain $\mathbb{A}^{1}$-bundles $\rho
:X\rightarrow \mathbb{A}_{\ast }^{2}$ from $\mathrm{SL}_{2}\left( \mathbb{C}%
\right) $. However, one checks easily using the explicit description in \ref%
{exa:LocalDesc} above that $\mathrm{ML}\left( X\right) =\mathbb{C}$ for
every non trivial $\mathbb{A}^{1}$-bundle $\rho :X\rightarrow \mathbb{A}%
_{\ast }^{2}$. Actually, if $X=X\left( m,n,p\right) $, where $p\in \mathbb{C}%
\left[ x,y\right] $ is a homogeneous polynomial, then Theorem \ref%
{thm:HomBunIsoType} below implies that the total space of $X$ is isomorphic
to $X_{1,r}=\left\{ xv-y^{r}u=1\right\} $, where $r+1=m+n-\deg p\geq 2$,
which is even a \emph{flexible variety}, i.e., the tangent space at every point
of $x$ is spanned by the tangent vectors to the orbits of the $\mathbb{G}_{a}
$-actions on $X$ \cite{ArzhKuyumZai2010}. We do not know if this additional property
holds for general $\mathbb{A}^{1}$-bundles $\rho :X\rightarrow \mathbb{A}%
_{\ast }^{2}$.
\end{rem}

\subsection{\label{sub:HomBunMain} Homogeneous $\mathbb{G}_{a}$-bundles}

\indent\newline
\noindent Here we develop the notion of homogeneous $\mathbb{G}_{a}$-bundle
over $\mathbb{A}_{\ast }^{2}$ that will be used in the rest of the article.

\begin{parn}
Let $\sigma :\mathbb{G}_{m}\times \mathbb{A}_{\ast }^{2}\rightarrow \mathbb{A%
}_{\ast }^{2}$ denote the linear $\mathbb{G}_{m}$-action on $\mathbb{A}%
_{\ast }^{2}$ with quotient $\pi :\mathbb{A}_{\ast }^{2}\rightarrow \mathbb{P%
}^{1}=\mathrm{Proj}\left( \mathbb{C}\left[ x,y\right] \right) $. Since $\pi $
is an affine morphism and $\pi _{\ast }\mathcal{O}_{\mathbb{A}_{\ast
}^{2}}\simeq \bigoplus_{d\in \mathbb{Z}}\mathcal{O}_{\mathbb{P}^{1}}\left(
-d\right) $, we have a decomposition 
\begin{equation*}
H^{1}(\mathbb{A}_{\ast }^{2},\mathcal{O}_{\mathbb{A}_{\ast }^{2}})\simeq
H^{1}(\mathbb{P}^{1},\pi _{\ast }\mathcal{O}_{\mathbb{A}_{\ast }^{2}})\simeq
\bigoplus_{d\in \mathbb{Z}}H^{1}(\mathbb{P}^{1},\mathcal{O}_{\mathbb{P}%
^{1}}\left( -d\right) )\simeq \bigoplus_{d\geq 2}H^{1}(\mathbb{P}^{1},%
\mathcal{O}_{\mathbb{P}^{1}}\left( -d\right) ),
\end{equation*}%
where for every $d$, $H^{1}(\mathbb{P}^{1},\mathcal{O}_{\mathbb{P}%
^{1}}\left( -d\right) )$ can be identified with the vector space of the
semi-invariants of weight $-d$ for the representation of $\mathbb{G}_{m}$ on 
$H^{1}(\mathbb{A}_{\ast }^{2},\mathcal{O}_{\mathbb{A}_{\ast }^{2}})$ induced
by $\sigma $. Via the one-to-one correspondence between coverings $\mathcal{V%
}=\left( V_{i}\right) _{i\in I}$ of $\mathbb{P}^{1}=\mathbb{A}_{\ast }^{2}/%
\mathbb{G}_{m}$ by affine open subsets $V_{i}$, $i\in I$ and coverings $%
\mathcal{U}=\left( U_{i}\right) _{i\in I}$ of $\mathbb{A}_{\ast }^{2}$ by $%
\mathbb{G}_{m}$-stable affine open subsets $U_{i}=\pi ^{-1}\left(
V_{i}\right) $, $i\in I$, a cohomology class in $H^{1}(\mathbb{P}^{1},\pi
_{\ast }\mathcal{O}_{\mathbb{A}_{\ast }^{2}})$ belongs to $H^{1}(\mathbb{P}%
^{1},\mathcal{O}_{\mathbb{P}^{1}}\left( -d\right) )$ if and only if it can
be represented by a \v{C}ech $1$-cocycle $\left\{ h_{ij}\right\} _{i,j\in
I}\in C^{1}(\mathcal{V},\mathcal{O}_{\mathbb{P}^{1}}\left( -d\right)
)\subset C^{1}(\mathcal{V},\pi _{\ast }\mathcal{O}_{\mathbb{A}_{\ast
}^{2}})\simeq C^{1}(\mathcal{U},\mathcal{O}_{\mathbb{A}_{\ast }^{2}})$
consisting of rational functions $h_{ij}\in \Gamma (U_{i}\cap U_{j},\mathcal{%
O}_{\mathbb{A}_{\ast }^{2}})\subset \mathbb{C}\left( x,y\right) $ that are
homogeneous of degree $-d$. These cocycles correspond precisely to $\mathbb{G%
}_{a}$-bundles $\tilde{\rho}:\tilde{X}\rightarrow \mathbb{A}_{\ast }^{2}$
with local trivializations $\tau _{i}:\tilde{X}\mid _{U_{i}}\overset{\sim }{%
\rightarrow }U_{i}\times \mathbb{G}_{a}$ for which the isomorphisms $\tau
_{i}\circ \tau _{j}^{-1}\mid _{U_{i}\cap U_{j}}$, $(u,t_{j})\mapsto
(u,t_{j}+h_{ij}\left( u\right) )$, $i,j\in I$, are equivariant for the
actions of $\mathbb{G}_{m}$ on $U_{i}\times \mathbb{G}_{a}$ and $U_{j}\times 
\mathbb{G}_{a}$ by $\lambda \cdot \left( u,t\right) =(\sigma \left( \lambda
,u\right) ,\lambda ^{-d}t)$. This leads to the following interpretation of
the above decomposition in terms of $\mathbb{G}_{a}$-bundles over $\mathbb{A}%
_{\ast }^{2}$.
\end{parn}

\begin{prop}
\label{pro:HomBunChar} For a $\mathbb{G}_{a}$-bundle $\rho:X\rightarrow%
\mathbb{A}_{*}^{2}$, the following are equivalent:

a) The isomorphy class of $\rho:X\rightarrow\mathbb{A}_{*}^{2}$ belongs to $%
H^{1}(\mathbb{P}^{1},\mathcal{O}_{\mathbb{P}^{1}}\left(-d\right))\subset
H^{1}(\mathbb{A}_{*}^{2},\mathcal{O}_{\mathbb{A}_{*}^{2}})$,

b) $\rho:X\rightarrow\mathbb{A}_{*}^{2}$ is isomorphic to a $\mathbb{G}_{a}$%
-bundle $\tilde{\rho}:\tilde{X}\rightarrow\mathbb{A}_{*}^{2}$ admitting a
lift $\tilde{\sigma}:\mathbb{G}_{m}\times\tilde{X}\rightarrow\tilde{X}$ of $%
\sigma$ for which there exists a collection of $\mathbb{G}_{a}$-equivariant
local trivializations $\tau_{i}:\tilde{X}\mid_{U_{i}}\overset{\sim}{%
\rightarrow}U_{i}\times\mathbb{G}_{a}$ over a covering of $\mathbb{A}%
_{*}^{2} $ by $\mathbb{G}_{m}$-stable affine open subsets $%
\left(U_{i}\right)_{i\in I} $ such that for every $i\in I$, $\tau_{i}$ is $%
\mathbb{G}_{m}$-equivariant for the action of $\mathbb{G}_{m}$ on $%
U_{i}\times\mathbb{G}_{a}$ defined by $\lambda\cdot\left(u,t\right)=\left(%
\sigma\left(\lambda,u\right),\lambda^{-d}t\right)$.
\end{prop}

\begin{defn}
A nontrivial $\mathbb{G}_{a}$-bundle $\rho:X\rightarrow\mathbb{A}_{*}^{2}$
satisfying one of the above equivalent properties for a certain $d\geq2$ is
said to be $-d$\emph{-homogeneous}.
\end{defn}

\begin{example}
\label{exa:Equiv-Can-cover} By specializing to the covering $\mathcal{U}_{0}$
of $\mathbb{A}_{\ast }^{2}$ by the $\mathbb{G}_{m}$-stable principal open
subsets $U_{x}$ and $U_{y}$, we obtain a more explicit decomposition: 
\begin{equation*}
H^{1}(\mathbb{A}_{\ast }^{2},\mathcal{O}_{\mathbb{A}_{\ast }^{2}})\simeq 
\check{H}^{1}(\mathcal{U}_{0},\mathcal{O}_{\mathbb{A}_{\ast }^{2}})\simeq
\bigoplus_{d\geq 2}W_{-d}
\end{equation*}
where, for every $d\geq 2$, $W_{-d}\simeq H^{1}(\mathbb{P}^{1},\mathcal{O}_{%
\mathbb{P}^{1}}\left( -d\right) )$ denotes the sub-$\mathbb{C}$-vector space
of $\mathbb{C}\left[ x^{\pm 1},y^{\pm 1}\right] $ with basis $\mathcal{B}%
_{d} $ consisting of rational monomials $x^{-m}y^{-n}$ where $m,n\in \mathbb{%
N}\setminus \{0\}$ and $m+n=d$. Note that letting $V_{d-2}\simeq H^{0}(%
\mathbb{P}^{1},\mathcal{O}_{\mathbb{P}^{1}}(d-2))$ be the space of binary
forms of degree $d-2$, Serre duality for $\mathbb{P}^{1}$ takes the form of
a perfect pairing $W_{-d}\times V_{d-2}\rightarrow W_{-2}$ for which the
basis $\mathcal{B}_{d}$ is simply the dual of the usual basis of $V_{d-2}$
consisting of monomials $x^{p}y^{q}$ with $p+q=d-2$. Therefore, every non
trivial class in $H^{1}(\mathbb{P}^{1},\mathcal{O}_{\mathbb{P}^{1}}(-d))$ is
represented by a \v{C}ech $1$-cocycle of the form $x^{-m}y^{-n}p\left(
x,y\right) \in C^{1}(\mathcal{U}_{0},\mathcal{O}_{\mathbb{A}_{\ast }^{2}})$,
where $p\left( x,y\right) \in \mathbb{C}\left[ x,y\right] $ is a homogeneous
polynomial of degree $r=m+n-d\geq 0$. The corresponding $\mathbb{G}_{a}$%
-bundle $\tilde{X}=X\left( m,n,p\right) =\left\{ x^{m}v-y^{n}u=p\left(
x,y\right) \right\} \setminus \left\{ x=y=0\right\} $ as in \ref%
{exa:LocalDesc} admits an obvious lift $\tilde{\sigma}\left( \lambda ,\left(
x,y,u,v\right) \right) =\left( \lambda x,\lambda y,\lambda ^{m-d}u,\lambda
^{n-d}v\right) $ of the $\mathbb{G}_{m}$-action $\sigma $ on $\mathbb{A}%
_{\ast }^{2}$ which satisfies b) in Proposition \ref{pro:HomBunChar} above.
\end{example}

\begin{parn}
\label{par:HomB-Princ} Proposition \ref{pro:HomBunChar} can be interpreted
from another point of view as a correspondence between $-d$-homogeneous $%
\mathbb{G}_{a}$-bundles $\rho :X\rightarrow \mathbb{A}_{\ast }^{2}$ and 
\emph{principal homogeneous bundles} $\nu :Y\rightarrow \mathbb{P}^{1}$ 
\emph{under the line bundle} $p:\mathcal{O}_{\mathbb{P}^{1}}\left( -d\right)
\rightarrow \mathbb{P}^{1}$. Here we consider the line bundle $p:\mathcal{O}%
_{\mathbb{P}^{1}}\left( -d\right) \rightarrow \mathbb{P}^{1}$ as equipped
with the structure of a locally constant group scheme over $\mathbb{P}^{1}$,
with group law induced by the diagonal homomorphism of sheaves $\mathcal{O}_{%
\mathbb{P}^{1}}\left( d\right) \rightarrow \mathcal{O}_{\mathbb{P}%
^{1}}\left( d\right) \oplus \mathcal{O}_{\mathbb{P}^{1}}\left( d\right) $. A
principal homogeneous $\mathcal{O}_{\mathbb{P}^{1}}\left( -d\right) $-bundle
(or simply an $\mathcal{O}_{\mathbb{P}^{1}}\left( -d\right) $-bundle) is a
scheme $\nu :Y\rightarrow \mathbb{P}^{1}$ equipped with an action of $%
\mathcal{O}_{\mathbb{P}^{1}}\left( -d\right) $ such that every point of $%
\mathbb{P}^{1}$ has a Zariski open neighborhood $U$ such that $Y\mid _{U}$
is equivariantly isomorphic to $\mathcal{O}_{\mathbb{P}^{1}}\left( -d\right)
\mid _{U}$ acting on itself by translations. Isomorphy classes of $\mathcal{O%
}_{\mathbb{P}^{1}}\left( -d\right) $-bundles are in one-to-one
correspondence with elements of the group $H^{1}(\mathbb{P}^{1},\mathcal{O}_{%
\mathbb{P}^{1}}\left( -d\right) )$. Example \ref{exa:HomBunDesc} below shows
that for a $\mathbb{G}_{a}$-bundle $\tilde{\rho}:\tilde{X}\rightarrow 
\mathbb{A}_{\ast }^{2}$ equipped with a lift $\tilde{\sigma}$ of $\sigma $
as in b), the quotient $\mathbb{A}^{1}$-bundle $\nu :\tilde{X}/\mathbb{G}%
_{m}\rightarrow \mathbb{P}^{1}=\mathbb{A}_{\ast }^{2}/\mathbb{G}_{m}$ comes
naturally equipped with the structure of principal homogeneous $\mathcal{O}_{%
\mathbb{P}^{1}}\left( -d\right) $-bundle with isomorphy class $\gamma =[%
\tilde{X}]\in H^{1}(\mathbb{P}^{1},\mathcal{O}_{\mathbb{P}^{1}}\left(
-d\right) )\subset H^{1}(\mathbb{A}_{\ast }^{2},\mathcal{O}_{\mathbb{A}%
_{\ast }^{2}})$.
\end{parn}

\begin{example}
\label{exa:HomBunDesc} By virtue of example \ref{exa:Equiv-Can-cover} above,
every nontrivial $-d$-homogeneous $\mathbb{G}_{a}$-bundle is isomorphic to
one of the form $\tilde{X}=X\left( m,n,p\right) $, where $p\left(
x,y\right) \in \mathbb{C}\left[ x,y\right] $ is homogeneous of degree $%
r=m+n-d\geq 0$. The latter is equipped with the lift $\tilde{\sigma}\left(
\lambda ,\left( x,y,u,v\right) \right) =\left( \lambda x,\lambda y,\lambda
^{m-d}u,\lambda ^{n-d}v\right) $ of $\sigma $ for which we have local $%
\mathbb{G}_{m}$-equivariant trivializations 
\begin{eqnarray*}
\tilde{X}\mid _{U_{x}}\simeq \mathrm{Spec}\left( \mathbb{C}\left[
x^{-1}y,x^{d-m}u\right] \right) \times \mathbb{G}_{m} &=&\mathrm{Spec}\left( 
\mathbb{C}\left[ z,w\right] \right) \times \mathbb{G}_{m} \\
\tilde{X}\mid _{U_{y}}\simeq \mathrm{Spec}\left( \mathbb{C}\left[
xy^{-1},y^{d-n}v\right] \right) \times \mathbb{G}_{m} &=&\mathrm{Spec}\left( 
\mathbb{C}\left[ z^{\prime },w^{\prime }\right] \right) \times \mathbb{G}%
_{m}.
\end{eqnarray*}%
These induce trivializations $\tau _{x}:\nu ^{-1}\left( U_{x}/\mathbb{G}%
_{m}\right) \overset{\sim }{\rightarrow }\mathrm{Spec}\left( \mathbb{C}\left[
z\right] \left[ w\right] \right) $ and $\tau _{y}:\nu ^{-1}\left( U_{y}/%
\mathbb{G}_{m}\right) \overset{\sim }{\rightarrow }\mathrm{Spec}\left( 
\mathbb{C}\left[ z^{\prime }\right] \left[ w^{\prime }\right] \right) $ of
the quotient bundle $\nu :\tilde{X}/\mathbb{G}_{m}\rightarrow \mathbb{P}^{1}$
for which the transition isomorphism $\tau _{y}\circ \tau _{x}^{-1}\mid
_{U_{x}\cap U_{x}/\mathbb{G}_{m}}$ has the form $\left( z,w\right) \mapsto
\left( z^{-1},z^{d}w+z^{m}p\left( z^{-1},1\right) \right) $. To see
explicitly the structure of an $O_{\mathbb{P}^{1}}(-d)$ bundle on $\tilde{X}/%
\mathbb{G}_{m},$ choose coordinates for the local trivializations of the
total space of $O_{\mathbb{P}^{1}}(-d)$ as follows : 
\begin{eqnarray*}
\mathcal{O}_{\mathbb{P}^{1}}\left( -d\right) \mid _{U_{x}/\mathbb{G}%
_{m}}\simeq \mathrm{Spec}\left( \mathbb{C}\left[ x^{-1}y,x^{d}t\right]
\right) &=&\mathrm{Spec}\left( \mathbb{C}\left[ z,\ell \right] \right) \\
\mathcal{O}_{\mathbb{P}^{1}}\left( -d\right) \mid _{U_{y}/\mathbb{G}%
_{m}}\simeq \mathrm{Spec}\left( \mathbb{C}\left[ xy^{-1},y^{d}t\right]
\right) &=&\mathrm{Spec}\left( \mathbb{C}\left[ z^{\prime },\ell ^{\prime }%
\right] \right) .
\end{eqnarray*}%
Then we see that the $\mathbb{G}_{a}$-action $t\cdot \left( x,y,u,v\right)
=\left( x,y,u+x^{m}t,v+y^{n}v\right) $ on $\tilde{X}$ descends to the action
of $\mathcal{O}_{\mathbb{P}^{1}}\left( -d\right) $ on $\tilde{X}/\mathbb{G}%
_{m}$ defined locally by $\ell \cdot w=w+\ell $ and $\ell ^{\prime }\cdot
w^{\prime }=w^{\prime }+\ell ^{\prime }$, which equips $\tilde{X}/\mathbb{G}%
_{m}$ with the structure of an $\mathcal{O}_{\mathbb{P}^{1}}\left( -d\right) 
$-bundle. Finally, with our choice of coordinate, the natural isomorphism of $
\mathbb{C}$-vectorspaces 
\begin{equation*}
\phi :C^{1}(\{U_{x},U_{y}\},\mathcal{O}_{\mathbb{A}_{\ast }^{2}})=\mathbb{C}%
\left[ x^{\pm 1},y^{\pm 1}\right] \overset{\sim }{\rightarrow }C^{1}(\{U_{x}/%
\mathbb{G}_{m},U_{y}/\mathbb{G}_{m}\},\pi _{\ast }\mathcal{O}_{\mathbb{A}%
_{\ast }^{2}})=\mathbb{C}\left[ z^{\pm 1}\right]
\end{equation*}%
maps a Laurent monomial $x^{i}y^{j}$ to $z^{-i}$, whence sends the \v{C}ech
cocycle $x^{-m}y^{-n}p\left( x,y\right) $ representing the isomorphy class
of the $\mathbb{G}_{a}$-bundle $\tilde{\rho}:\tilde{X}\rightarrow \mathbb{A}%
_{\ast }^{2}$ to the one $z^{m}p\left( z^{-1},1\right) $ representing the
isomorphy class of the $\mathcal{O}_{\mathbb{P}^{1}}\left( -d\right) $%
-bundle $\nu :\tilde{X}/\mathbb{G}_{m}\rightarrow \mathbb{P}^{1}$.
\end{example}

\begin{rem}
The precise correspondence between $-d$-homogeneous $\mathbb{G}_{a}$-bundles
over $\mathbb{A}_{\ast }^{2}$ and principal $\mathcal{O}_{\mathbb{P}%
^{1}}\left( -d\right) $-bundles takes the form of an equivalence of
categories extending the one between $\mathbb{G}_{m}$-linearized line
bundles on $\mathbb{A}_{\ast }^{2}$ with respect to the action 
$\sigma : \mathbb{G}_m \times \mathbb{A}_{\ast }^{2} \rightarrow \mathbb{A}_{\ast }^{2}$ 
and line bundles over $\mathbb{P}^{1}=\mathbb{A}_{\ast }^{2}/\mathbb{G}_{m}$. 
Recall that a $\mathbb{G}_{m}$-linearized line bundle is a pair $(L,\Phi )$ consisting of a line bundle 
$p:L\rightarrow \mathbb{A}_{\ast }^{2}$ and an isomorphism 
$$\Phi :\sigma
^{\ast }L=(\mathbb{G}_m\times \mathbb{A}_{\ast }^{2})\times_{\sigma, \mathbb{A}_{\ast }^{2}} L\overset{\sim }{\longrightarrow } \mathrm{p}_{2}^{\ast }L=(\mathbb{G}_m\times \mathbb{A}_{\ast }^{2})\times_{\mathrm{p}_2, \mathbb{A}_{\ast }^{2}} L$$  of line
bundles over $\mathbb{G}_{m}\times \mathbb{A}_{\ast }^{2}$ satisfying the
cocycle condition $(\mu \times \mathrm{id}_{\mathbb{A}_{\ast }^{2}})^{\ast
}\Phi =\mathrm{p}_{23}^{\ast }\Phi \circ (\mathrm{id}_{\mathbb{G}_{m}}\times
\sigma )^{\ast }\Phi $ over $\mathbb{G}_{m}\times \mathbb{G}_{m}\times 
\mathbb{A}_{\ast }^{2}$, where $\mu :\mathbb{G}_{m}\times \mathbb{G}%
_{m}\rightarrow \mathbb{G}_{m}$ denotes the group law of $\mathbb{G}_{m}$
(see e.g. \cite[§3, p.30]{Mumford1994}). A standard argument of faithfully
flat descent for the quotient morphism $\pi :\mathbb{A}_{\ast
}^{2}\rightarrow \mathbb{P}^{1}=\mathbb{A}_{\ast }^{2}/\mathbb{G}_{m}$ shows
that the category of $\mathbb{G}_{m}$-linearized line bundles over $\mathbb{A%
}_{\ast }^{2}$ is equivalent to category of line bundles over $\mathbb{P}%
^{1} $. Noting that $\Phi :\sigma ^{\ast }L\overset{\sim }{\rightarrow }%
\mathrm{p}_{2}^{\ast }L$ is an isomorphism of group schemes over $\mathbb{G}%
_{m}\times \mathbb{A}_{\ast }^{2}$, we can define a category $\tilde{\mathbb{%
V}}$ whose objects are pairs $\{(L,\Phi ),(\tilde{X},\Psi )\}$ consisting of
a $\mathbb{G}_{m}$-linearized line bundle $(L,\Phi )$, a principal $L$%
-bundle $\tilde{\rho}:\tilde{X}\rightarrow \mathbb{A}_{\ast }^{2}$, and a $%
\Phi $-equivariant isomorphism $\Psi :\sigma ^{\ast }\tilde{X}\overset{\sim }%
{\rightarrow }\mathrm{p}_{2}^{\ast }\tilde{X}$ of principal bundles under $%
\sigma ^{\ast }L$ and $\mathrm{p}_{2}^{\ast }L$ respectively, satisfying the
cocycle condition $(\mu \times \mathrm{id}_{\mathbb{A}_{\ast }^{2}})^{\ast
}\Psi =\mathrm{p}_{23}^{\ast }\Psi \circ (\mathrm{id}_{\mathbb{G}_{m}}\times
\sigma )^{\ast }\Psi $ over $\mathbb{G}_{m}\times \mathbb{G}_{m}\times 
\mathbb{A}_{\ast }^{2}$. Then one checks that the previous equivalence
extends to a one between $\tilde{\mathbb{V}}$ and the category $\mathbb{V}$
whose objects are pairs $(M,Y)$ consisting of a line bundle $q:M\rightarrow 
\mathbb{P}^{1}$ and a principal $M$-bundle $\nu :Y\rightarrow \mathbb{P}^{1}$%
.

Recall that for a $\mathbb{G}_{m}$-linearized line bundle $(L,\Phi )$ over $%
\mathbb{A}_{\ast }^{2}$ the morphism $\overline{\sigma }=\mathrm{p}_{2}\circ
\Phi ^{-1}:\mathbb{G}_{m}\times L\simeq \sigma ^{\ast }L\rightarrow L$
defines a lift to $L$ of the $\mathbb{G}_{m}$-action $\sigma $ on $\mathbb{A}%
_{\ast }^{2}$ which is {}\textquotedblleft linear on the
fibers\textquotedblright\ of $p:L\rightarrow \mathbb{A}_{\ast }^{2}$.
Similarly, for $(\tilde{X},\Psi )$ as above, the morphism $\tilde{\sigma}=%
\mathrm{p}_{2}\circ \Psi ^{-1}:\mathbb{G}_{m}\times \tilde{X}\simeq \sigma
^{\ast }\tilde{X}\rightarrow \tilde{X}$ is a lift to $\tilde{X}$ of $\sigma $
for which $\tilde{X}$ {}\textquotedblleft locally looks like $L$ equipped
with the action $\overline{\sigma }$\textquotedblright . By specializing to
the case of the trivial line bundle $\mathbb{A}_{\ast }^{2}\times \mathrm{%
Spec}\left( \mathbb{C}\left[ t\right] \right) $ over $\mathbb{A}_{\ast }^{2}$
equipped with the $\mathbb{G}_{m}$-linearization given by the lift $\lambda
\cdot \left( x,y,\ell \right) =\left( \lambda x,\lambda y,\lambda
^{-d}t\right) $ of $\sigma $, which corresponds to the line bundle $\mathcal{%
O}_{\mathbb{P}^{1}}\left( -d\right) \rightarrow \mathbb{P}^{1}$, the
previous equivalence boils down to a one-to-one correspondence between
isomorphy classes of principal $\mathcal{O}_{\mathbb{P}^{1}}\left( -d\right) 
$-bundles over $\mathbb{P}^{1}$ and isomorphy classes of $\mathbb{G}_{a}$%
-bundles $\tilde{\rho}:\tilde{X}\rightarrow \mathbb{A}_{\ast }^{2}$ equipped
with a lift $\tilde{\sigma}:\mathbb{G}_{m}\times \tilde{X}\rightarrow \tilde{%
X}$ of $\sigma $ as in b) in Proposition \ref{pro:HomBunChar}.
\end{rem}

\section{Isomorphy types of total spaces of $\mathbb{A}^{1}$-bundles over $%
\mathbb{A}_{*}^{2}$}

In this section, we give partial answers to the problem of classifying total
spaces of nontrivial $\mathbb{A}^{1}$-bundles over $\mathbb{A}_{*}^{2}$
considered as abstract affine varieties.

\subsection{Base change under the action of $\mathrm{Aut}\left(\mathbb{A}%
_{*}^{2}\right)$}

\indent\newline
\noindent Since we are interested in isomorphy types of total spaces of $%
\mathbb{A}^{1}$-bundles over $\mathbb{A}_{\ast }^{2}$ as abstract varieties,
regardless of the particular $\mathbb{A}^{1}$-bundle structure, a natural
step is to consider these bundles up to a weaker notion of bundle
isomorphism which consists in identifying two nontrivial $\mathbb{A}^{1}$%
-bundles $\rho :X\rightarrow \mathbb{A}_{\ast }^{2}$ and $\rho ^{\prime
}:X^{\prime }\rightarrow \mathbb{A}_{\ast }^{2}$ if there exists a
commutative diagram 
\begin{equation*}
\xymatrix{X' \ar[d]_{\rho'} \ar[r]^{\Psi} & X \ar[d]^{\rho} \\
\mathbb{A}^2_* \ar[r]^{\psi} & \mathbb{A}^2_*}
\end{equation*}%
where $\psi $ and $\Psi $ are isomorphisms. This means equivalently that the
isomorphy classes of $X$ and $X^{\prime }$ in $\mathbb{P}H^{1}(\mathbb{A}%
_{\ast }^{2},\mathcal{O}_{\mathbb{A}_{\ast }^{2}})$ belong to the same orbit
of the action of the group $\mathrm{Aut}\left( \mathbb{A}_{\ast }^{2}\right) 
$ of automorphisms of $\mathbb{A}_{\ast }^{2}$ on $\mathbb{P}H^{1}(\mathbb{A}%
_{\ast }^{2},\mathcal{O}_{\mathbb{A}_{\ast }^{2}})$ induced by the linear
representation 
\begin{equation*}
\eta :\mathrm{Aut}\left( \mathbb{A}_{\ast }^{2}\right) \rightarrow \mathrm{GL%
}(H^{1}(\mathbb{A}_{\ast }^{2},\mathcal{O}_{\mathbb{A}_{\ast }^{2}})),\;\psi
\mapsto \eta \left( \psi \right) =\psi ^{\ast }:H^{1}(\mathbb{A}_{\ast }^{2},%
\mathcal{O}_{\mathbb{A}_{\ast }^{2}})\overset{\sim }{\rightarrow }H^{1}(%
\mathbb{A}_{\ast }^{2},\mathcal{O}_{\mathbb{A}_{\ast }^{2}})
\end{equation*}%
of $\mathrm{Aut}\left( \mathbb{A}_{\ast }^{2}\right) $ on $H^{1}(\mathbb{A}%
_{\ast }^{2},\mathcal{O}_{\mathbb{A}_{\ast }^{2}})$, where $\psi ^{\ast }$
maps the isomorphy class of $\mathbb{G}_{a}$-bundle $\rho :X\rightarrow 
\mathbb{A}_{\ast }^{2}$ to that of the $\mathbb{G}_{a}$-bundle $\mathrm{pr}%
_{2}:X\times _{\rho ,\mathbb{A}_{\ast }^{2},\psi }\mathbb{A}_{\ast
}^{2}\rightarrow \mathbb{A}_{\ast }^{2}$.

\begin{parn}
\label{par:GL2-rep} The group of automorphisms of $\mathbb{A}_{\ast }^{2}$
can be identified with the subgroup of $\mathrm{Aut}\left( \mathbb{A}%
^{2}\right) $ consisting of automorphisms of the plane $\mathbb{A}^{2}$ that
preserve the origin $o$. As a consequence of Jung's Theorem \cite{Jung1942}, 
$\mathrm{Aut}\left( \mathbb{A}_{\ast }^{2}\right) $ is generated by the
general linear group $\mathrm{GL}_{2}\left( \mathbb{C}\right) $ and the
subgroup $U\subset \mathrm{Aut}\left( \mathbb{A}_{\ast }^{2}\right) $
consisting of automorphisms of the form $\left( x,y\right) \mapsto \left(
x,y+p\left( x\right) \right) $ where $p\left( x\right) \in x^{2}\mathbb{C}%
\left[ x\right] $. Since the representation of $\mathbb{G}_{m}$ on $H^{1}(%
\mathbb{A}_{\ast }^{2},\mathcal{O}_{\mathbb{A}_{\ast }^{2}})$ induced by the
action $\sigma :\mathbb{G}_{m}\times \mathbb{A}_{\ast }^{2}\rightarrow 
\mathbb{A}_{\ast }^{2}$ commutes with that of $\mathrm{GL}_{2}\left( \mathbb{%
C}\right) $, the decomposition 
\begin{equation*}
H^{1}(\mathbb{A}_{\ast }^{2},\mathcal{O}_{\mathbb{A}_{\ast }^{2}})\simeq
\bigoplus_{d\geq 2}H^{1}(\mathbb{P}^{1},\mathcal{O}_{\mathbb{P}^{1}}\left(
-d\right) )
\end{equation*}%
provides a splitting of the induced representation $\mathrm{GL}_{2}\left( 
\mathbb{C}\right) \rightarrow \mathrm{GL}(H^{1}(\mathbb{A}_{\ast }^{2},%
\mathcal{O}_{\mathbb{A}_{\ast }^{2}}))$ into a direct sum of representations
on the finite dimensional vector spaces $H^{1}(\mathbb{P}^{1},\mathcal{O}_{%
\mathbb{P}^{1}}\left( -d\right) )$, $d\geq 2$. Using the identifications $%
H^{1}(\mathbb{P}^{1},\mathcal{O}_{\mathbb{P}^{1}}\left( -d\right) )\simeq
W_{-d}$ and $H^{0}(\mathbb{P}^{1},\mathcal{O}_{\mathbb{P}^{1}}\left(
d-2\right) )\simeq V_{d-2}$ as in example \ref{exa:Equiv-Can-cover}, the
perfect pairing $W_{-d}\times V_{d-2}\rightarrow W_{-2}$ given by Serre
duality for $\mathbb{P}^{1}$ yields an isomorphism of representations $%
W_{-d}\simeq V_{d-2}^{\ast }\otimes W_{-2},$ where $\mathrm{GL}_{2}\left( 
\mathbb{C}\right) $ acts on the vector space $V_{d-2}$ of binary forms of
degree $d-2$ via the standard representation and on $W_{-2}\simeq H^{1}(%
\mathbb{P}^{1},\mathcal{O}_{\mathbb{P}^{1}}\left( -2\right) )\simeq \mathbb{C%
}$ by the inverse of the determinant.
\end{parn}

\begin{parn}
Since triangular automorphisms in $U$ do not preserve the usual degree on $%
\mathbb{C}\left[x,y\right]$, the induced representation $U\rightarrow\mathrm{%
GL(H^{1}(\mathbb{A}_{*}^{2},\mathcal{O}_{\mathbb{A}_{*}^{2}}))}$ does not
preserve the above decomposition of $H^{1}(\mathbb{A}_{*}^{2},\mathcal{O}_{%
\mathbb{A}_{*}^{2}})$ into a direct sum. However, letting $%
F_{-d}=\bigoplus_{i=2}^{d}W_{-i}\subset H^{1}(\mathbb{A}_{*}^{2},\mathcal{O}%
_{\mathbb{A}_{*}^{2}})$, $d\geq2$, we have the following description.
\end{parn}

\begin{lem}
\label{lem:Unip-Rep} For every $d\geq2$, the subspace $F_{-d}$ is $U$-stable
and the quotient representation on $F_{-\left(d+1\right)}/F_{-d}\simeq
W_{-\left(d+1\right)}$ is the trivial one.
\end{lem}

\begin{proof}
It is equivalent to show that the subspace $Q_{-d}^{\ast }$ of the dual of $%
H^{1}(\mathbb{A}_{\ast }^{2},\mathcal{O}_{\mathbb{A}_{\ast }^{2}})$ that is
orthogonal to $F_{-d}$ is stable under the dual representation, and that the
quotient representation on $Q_{-d}^{\ast }/Q_{-\left( d+1\right) }^{\ast }$
is the trivial one. By Serre duality again, we have 
\begin{equation*}
\left( H^{1}(\mathbb{A}_{\ast }^{2},\mathcal{O}_{\mathbb{A}_{\ast
}^{2}})\right) ^{\ast }\simeq (\bigoplus_{i\geq 2}W_{-i})^{\ast }\simeq
\prod_{i\geq 2}W_{-i}^{\ast }\simeq \prod_{i\geq 2}V_{i-2},
\end{equation*}%
the dual representation on $\prod_{i\geq 2}V_{i-2}$ being induced by the
action of $U$ on $\mathbb{C}\left[ x,y\right] $ defined by $u\cdot p\left(
x,y\right) =p\left( u^{-1}\left( x,y\right) \right) $. For an element $%
u=\left( x,y+x^{2}s\left( x\right) \right) \in U$ and a homogeneous
polynomial $p_{n}\left( x,y\right) \in V_{n}$ of degree $n\geq 0$, one has $%
u\cdot p_{n}\left( x,y\right) =p_{n}\left( x,y\right) +R\left( x,y\right) $,
where $R$ is a finite sum of homogeneous polynomials of degrees $>n$. This
implies that 
\begin{equation*}
Q_{-d}^{\ast }=\prod_{i>d}V_{i-2}\subset \prod_{i\geq 2}V_{i-2}
\end{equation*}%
is $U$-stable and that the quotient representation on $Q_{-d}^{\ast
}/Q_{-\left( d+1\right) }^{\ast }\simeq V_{d-2}$ is the trivial one, as
desired.
\end{proof}

\begin{parn}
One cannot expect to have a general effective criterion to decide which
isomorphy classes of $\mathbb{G}_{a}$-bundles or $\mathbb{A}^{1}$-bundles
over $\mathbb{A}_{\ast }^{2}$ belong to the same orbit of the actions of $%
\mathrm{Aut}\left( \mathbb{A}_{\ast }^{2}\right) $ on $H^{1}(\mathbb{A}%
_{\ast }^{2},\mathcal{O}_{\mathbb{A}_{\ast }^{2}})$ and $\mathbb{P}H^{1}(%
\mathbb{A}_{\ast }^{2},\mathcal{O}_{\mathbb{A}_{\ast }^{2}})$ respectively.
But the above description provides at least strong restrictions for certain
homogeneous $\mathbb{G}_{a}$-bundles to be obtained as pull-backs of other
ones by an automorphism of $\mathbb{A}_{\ast }^{2}$. For instance, the
isomorphy class of the $\mathbb{A}^{1}$-bundle $\mathrm{pr}_{x,y}:\mathrm{SL}%
_{2}\left( \mathbb{C}\right) =\left\{ xv-yu=1\right\} \rightarrow \mathbb{A}%
_{\ast }^{2}$ is a fixed point of the projective representation of $\mathrm{%
Aut}\left( \mathbb{A}_{\ast }^{2}\right) $ on $\mathbb{P}H^{1}(\mathbb{A}%
_{\ast }^{2},\mathcal{O}_{\mathbb{A}_{\ast }^{2}})$, whence is stable under
arbitrary base change by an automorphism of $\mathbb{A}_{\ast }^{2}$. In the
same spirit, for the isomorphy classes of the homogeneous $\mathbb{G}_{a}$%
-bundles 
\begin{equation*}
\mathrm{pr}_{x,y}:X_{m,n}=X\left( m,n,1\right) =\left\{
x^{m}v-y^{n}u=1\right\} \rightarrow \mathbb{A}_{\ast }^{2},
\end{equation*}%
we have following result (compare with Theorem \ref{thm:HomBunIsoType} and
example \ref{exa:HomBundExplIso} below).
\end{parn}

\begin{prop}
The $\mathbb{A}^{1}$-bundles $X_{m,n}\rightarrow\mathbb{A}_{*}^{2}$ and $%
X_{p,q}\rightarrow\mathbb{A}_{*}^{2}$ can be obtained from each other by a
base change $\psi:\mathbb{A}_{*}^{2}\overset{\sim}{\rightarrow}\mathbb{A}%
_{*}^{2}$ if and only if $\left\{ m,n\right\} =\left\{ p,q\right\} $.
\end{prop}

\begin{proof}
It is equivalent to show that for $a\in \mathbb{C}^{\ast }$ the isomorphy
classes in $H^{1}(\mathbb{A}_{\ast }^{2},\mathcal{O}_{\mathbb{A}_{\ast
}^{2}})$ of $X_{p,q}$ and $X_{m,n}\left( a\right) =X\left( m,n,a\right) $
belong to the same orbit of the action of $\mathrm{Aut}\left( \mathbb{A}%
_{\ast }^{2}\right) $ if and only if $\left\{ m,n\right\} =\left\{
p,q\right\} $. Since $X_{m,n}\left( a\right) $ and $X_{p,q}$ are homogeneous
of degree $-m-n$ and $-p-q$, it follows from Proposition \ref{pro:HomBunChar}
and Lemma \ref{lem:Unip-Rep} that their isomorphy classes $[X_{m,n}(a)]\in
W_{-\left( m+n\right) }$ and $[X_{p,q}]\in W_{-\left( p+q\right) }$ belong
to the same orbit of $\mathrm{Aut}\left( \mathbb{A}_{\ast }^{2}\right) $ if
and only $m+n=p+q=d$ and they belong to the same orbit of the action of $%
\mathrm{GL}_{2}\left( \mathbb{C}\right) $ on $W_{-d}$. By duality, this
holds if and only if the homogeneous polynomials $a^{-1}x^{m-1}y^{n-1}$ and $%
x^{p-1}y^{q-1}$ belong to the same orbit of the standard representation of $%
\mathrm{GL}_{2}\left( \mathbb{C}\right) $ on $V_{d-2}$, which is the case if
and only if $\left\{ m-1,n-1\right\} =\left\{ p-1,q-1\right\} $.
\end{proof}

\subsection{Isomorphy types of homogeneous $\mathbb{G}_{a}$-bundles}

\indent\newline
\noindent Recall \ref{sub:HomBunMain} that a nontrivial $\mathbb{G}_{a}$%
-bundle $\rho:X\rightarrow\mathbb{A}_{*}^{2}$ is called $-d$-homogeneous if
it represents a cohomology class in $H^{1}(\mathbb{P}^{1},\mathcal{O}_{%
\mathbb{P}^{1}}\left(-d\right))\subset H^{1}(\mathbb{A}_{*}^{2},\mathcal{O}_{%
\mathbb{A}_{*}^{2}})$ for a certain $d\geq2$. All these bundles come
equipped with a lift of the $\mathbb{G}_{m}$-action $\sigma:\mathbb{G}%
_{m}\times\mathbb{A}_{*}^{2}\rightarrow\mathbb{A}_{*}^{2}$ as in Proposition %
\ref{pro:HomBunChar}, and we have the following characterization:

\begin{thm}
\label{thm:HomBunIsoType} The total spaces of two nontrivial homogeneous $%
\mathbb{G}_{a}$-bundles are $\mathbb{G}_{m}$-equivariantly isomorphic if and
only if they have the same homogeneous degree. In particular, for a fixed $%
d\geq2,$ the total spaces of nontrivial $-d$-homogeneous $\mathbb{G}_{a}$%
-bundles $\rho:X\rightarrow\mathbb{A}_{*}^{2}$ are all isomorphic as
abstract affine varieties.
\end{thm}

\begin{proof}
Suppose that $\rho :X\rightarrow \mathbb{A}_{\ast }^{2}$ and $\rho ^{\prime
}:X^{\prime }\rightarrow \mathbb{A}_{\ast }^{2}$ are homogeneous of degrees $%
-d$ and $-d^{\prime }$ respectively. A $\mathbb{G}_{m}$-equivariant
isomorphism between $X^{\prime }$ and $X$ descends to an isomorphism $%
f:X^{\prime }/\mathbb{G}_{m}\overset{\sim }{\rightarrow }X/\mathbb{G}_{m}$
between the total space of the principal $\mathcal{O}_{\mathbb{P}^{1}}\left(
-d\right) $-bundle $\overline{\rho }:X/\mathbb{G}_{m}\rightarrow \mathbb{P}%
^{1}$ and that of the principal $\mathcal{O}_{\mathbb{P}^{1}}\left(
-d^{\prime }\right) $-bundle $\overline{\rho }^{\prime }:X^{\prime }/\mathbb{%
G}_{m}\rightarrow \mathbb{P}^{1}$ (see \ref{par:HomB-Princ}). The Picard
group of a principal $\mathcal{O}_{\mathbb{P}^{1}}\left( -i\right) $-bundle $%
\nu :Y\rightarrow \mathbb{P}^{1}$ is isomorphic to $\mathbb{Z}$, generated
for instance by the pull-back $\nu ^{\ast }\mathcal{O}_{\mathbb{P}%
^{1}}\left( -1\right) $ of $\mathcal{O}_{\mathbb{P}^{1}}\left( -1\right) $.
Furthermore it follows from the exact sequence 
\begin{equation*}
0\rightarrow \nu ^{\ast }\Omega _{\mathbb{P}^{1}}^{1}\rightarrow \Omega
_{Y}^{1}\rightarrow \Omega _{Y/\mathbb{P}^{1}}^{1}\simeq \nu ^{\ast }%
\mathcal{O}_{\mathbb{P}^{1}}\left( i\right) \rightarrow 0
\end{equation*}%
that $\omega _{Y}=\Lambda ^{2}\Omega _{Y}^{1}\simeq \nu ^{\ast }\mathcal{O}_{%
\mathbb{P}^{1}}\left( i-2\right) $. Since the isomorphism $f^{\ast }:\mathrm{%
Pic}(X/\mathbb{G}_{m})\overset{\sim }{\rightarrow }\mathrm{Pic}(X^{\prime }/%
\mathbb{G}_{m})$ induced by $f$ sends $\omega _{X/\mathbb{G}_{m}}$ to $%
\omega _{X^{\prime }/\mathbb{G}_{m}}$, we conclude that $d=d^{\prime }$
necessarily. Suppose conversely that $\rho :X\rightarrow \mathbb{A}_{\ast
}^{2}$ and $\rho ^{\prime }:X^{\prime }\rightarrow \mathbb{A}_{\ast }^{2}$
are homogeneous of the same degree $-d\leq -2$. The existence of a $\mathbb{G}%
_{m}$-equivariant isomorphism between $X^{\prime }$ and $X$ is equivalent to
the existence of an isomorphism $f:X^{\prime }/\mathbb{G}_{m}\overset{\sim }{%
\rightarrow }X/\mathbb{G}_{m}$ for which $\pi ^{\prime }:X^{\prime
}\rightarrow X^{\prime }/\mathbb{G}_{m}$ and $\mathrm{p}_{2}:X\times _{X/%
\mathbb{G}_{m}}X^{\prime }/\mathbb{G}_{m}\rightarrow X^{\prime }/\mathbb{G}%
_{m}$ are isomorphic as $\mathbb{G}_{m}$-bundles over $X^{\prime }/\mathbb{G}%
_{m}$. Since the diagram 
\begin{equation*}
\xymatrix{ X \ar[d]_{\rho} \ar[r] & X/\mathbb{G}_m \ar[d]^{\overline{\rho}}
\\ \mathbb{A}^2_* \ar[r] & \mathbb{P}^1=\mathbb{A}^2_*/\mathbb{G}_m,}
\end{equation*}%
is Cartesian, the isomorphy class of the $\mathbb{G}_{m}$-bundle $%
X\rightarrow X/\mathbb{G}_{m}$ in $H^{1}(X/\mathbb{G}_{m},\mathbb{G}%
_{m})\simeq \mathrm{Pic}(X/\mathbb{G}_{m})$ coincides with $\overline{\rho }%
^{\ast }\mathcal{O}_{\mathbb{P}^{1}}\left( -1\right) $, and similarly for $%
X^{\prime }\rightarrow X^{\prime }/\mathbb{G}_{m}$. So the $\mathbb{G}_{m}$%
-equivariant isomorphy of $X$ and $X^{\prime }$ reduces to the existence of
an isomorphism $f:X^{\prime }/\mathbb{G}_{m}\overset{\sim }{\rightarrow }X/%
\mathbb{G}_{m}$ such that $f^{\ast }\overline{\rho }^{\ast }\mathcal{O}_{%
\mathbb{P}^{1}}\left( -1\right) =\overline{\rho ^{\prime }}^{\ast }\mathcal{O%
}_{\mathbb{P}^{1}}\left( -1\right) $. Since $\overline{\rho }:X/\mathbb{G}%
_{m}\rightarrow \mathbb{P}^{1}$ and $\overline{\rho }^{\prime }:X^{\prime }/%
\mathbb{G}_{m}\rightarrow \mathbb{P}^{1}$ are both nontrivial $\mathcal{O}_{%
\mathbb{P}^{1}}\left( -d\right) $-bundles, the Danilov-Gizatullin Theorem 
\cite[Theorem 5.8.1 and Remark 4.8.6]{Gizatullin1977} implies directly the
existence of an isomorphism $f:X^{\prime }/\mathbb{G}_{m}\overset{\sim }{%
\rightarrow }X/\mathbb{G}_{m}$. Moreover, since $f^{\ast }$ maps generators
of $\mathrm{Pic}(X/\mathbb{G}_{m})$ to generators of $\mathrm{Pic}(X^{\prime
}/\mathbb{G}_{m})$, it follows that $X^{\prime }$ and $X\times _{X/\mathbb{G}%
_{m}}X^{\prime }/\mathbb{G}_{m}$ are isomorphic as locally trivial $\mathbb{A%
}_{\ast }^{1}$-bundles over $X^{\prime }/\mathbb{G}_{m}$, which implies the
second assertion of the theorem. Finally, the existence of an isomorphism $f$
with the required additional property is guaranteed by Theorem \ref%
{thm:DanGizIso} in the Appendix.
\end{proof}

\begin{example}
\label{exa:HomBundExplIso} As a consequence of Theorem \ref%
{thm:HomBunIsoType} above, we get in particular that if $m+n=m^{\prime
}+n^{\prime }$, then the varieties $X_{m,n}$ and $X_{m^{\prime },n^{\prime
}} $ are isomorphic. While it appears to be rather difficult to construct an
explicit isomorphism between even the first interesting examples $X_{2,2}$
and $X_{3,1}$, one can check that the morphism 
\begin{equation*}
\pi:X_{2,2}=\left\{ x^{2}v-y^{2}u=1\right\} \rightarrow\mathbb{A}_{*}^{2}=%
\mathrm{Spec}\left(\mathbb{C}\left[a,b\right]\right)\setminus\left\{
\left(0,0\right)\right\} ,\;\left(x,y,u,v\right)\mapsto(x-\frac{1}{2}y,\frac{%
6x-y}{8}v-\frac{3y-2x}{2}u)
\end{equation*}
is $\mathbb{A}^{1}$-bundle isomorphic to the one $\rho:X_{3,1}\simeq\left\{
a^{3}v^{\prime }-bu^{\prime }=1\right\} \rightarrow\mathbb{A}_{*}^{2}$. More
precisely, letting 
\begin{equation*}
w=\frac{5}{16}v^{2}x+u^{2}x+\frac{5}{2}vux-\frac{1}{32}v^{2}y-\frac{5}{2}%
u^{2}y-\frac{5}{4}vuy\in\Gamma(X_{2,2},\mathcal{O}_{X_{2,2}}),
\end{equation*}
a direct computation shows that $\pi^{-1}(U_{a})\simeq U_{a}\times\mathrm{%
Spec}(\mathbb{C}\left[a^{-3}(y+a+ab)\right])$, $\pi^{-1}\left(U_{b}\right)%
\simeq U_{b}\times\mathrm{Spec}(\mathbb{C}\left[b^{-1}w\right])$ and that $%
a^{-3}\left(y+a+ab\right)-b^{-1}w=a^{-3}b^{-1}\in\Gamma(\pi^{-1}(U_{a}\cap
U_{b}),\mathcal{O}_{X_{2,2}})$. Thus $\pi:X_{2,2}\rightarrow\mathbb{A}%
_{*}^{2}$ is an $\mathbb{A}^{1}$-bundle with the same associated \v{C}ech
cocycle $a^{-3}b^{-1}\in C^{1}(\{U_{a},U_{b}\},\mathcal{O}_{\mathbb{A}%
_{*}^{2}})$ as $\rho:X_{3,1}\rightarrow\mathbb{A}_{*}^{2}$.
\end{example}

\subsection{Existence of exotic affine spheres}
\indent\newline\noindent Here we show that exotic affine spheres occur 
among the total spaces of non trivial $\mathbb{A}^{1}$-bundles over $\mathbb{A}
_{*}^{2}$. 
\begin{parn} To illustrate the idea behind the proof of Theorem \ref{thm:DeRhamArg} below, let us
first consider the varieties $X_{m,1}=\{x^mv-yu=1\}$, $m\geq 1$. Because $X_{m,1}$ is smooth, 
the canonical sheaf $\omega_{m}=\omega_{X_{m,1}}$ of $X_{m,1}$ is a free $\mathcal{O}_{X_{m,1}}$-module 
generated for instance by the global nowhere vanishing 3-form  
\begin{equation*}
\alpha _{m} = x^{-m}dx\wedge dy\wedge du\mid_{X_{m,1}}=-y^{-1}dx\wedge
dy\wedge dv\mid_{X_{m,1}}.
\end{equation*}
The pull back of $\alpha_1$  by an isomorphism $\varphi:X_{m,1}\stackrel{\sim}{\rightarrow} X_{1,1}$ 
would be a nowhere vanishing algebraic 3-form on $X_{m,1}$ whence a nonzero scalar multiple of $\alpha_m$ 
since nonzero constants are the only invertible functions $X_{m,1}$.
On the other hand, since $X_{m,1}$ has the real sphere $S^3$ 
as a strong deformation retract, the de Rham cohomology group $H_{dR}^{3}(X_{m,1},\mathbb{C})$ 
is one dimensional over $\mathbb{C}$. Using the fact that the de Rham cohomology of a smooth complex affine variety
equals the cohomology of its algebraic de Rham complex \cite{Grothendieck1966}, it can be
checked directly that $H_{dR}^{3}(X_{m,1},\mathbb{C})\simeq \Omega^3_{X_{m,1}}/d\Omega^2_{X_{m,1}}$ is spanned by the class of 
$x^{m-1}\alpha _{m}=x^{-1}dx\wedge dy\wedge du\mid_{X_{m,1}}$. The isomorphism $\varphi$ would induce an isomorphism in cohomology 
and since $H_{dR}^{3}(X_{1,1},\mathbb{C})$ is spanned by the class of $\alpha_1$ it would follow that  
$H_{dR}^{3}(X_{m,1},\mathbb{C})$ is spanned by the class of $\alpha_m$ too. This is absurd since for every $m\geq 2$, 
$\alpha _{m}$ is an exact form, having for instance the global 2-form 
$$\frac{dy\wedge du}{(1-m)x^{m-1}}\mid_{X_{m,1}}=\frac{xdy\wedge dv-mvdx\wedge
dy}{(1-m)y}\mid_{X_{m,1}}$$ as a primitive.
\end{parn}
\begin{parn}
\noindent More generally, recall that every non trivial $\mathbb{A}^{1}$%
-bundle $\rho :X\rightarrow \mathbb{A}_{\ast }^{2}$ is isomorphic to one of
the form 
\begin{equation*}
X\left( m,n,p\right) =\left\{ x^{m}v-y^{n}u=p\left( x,y\right) \right\}
\setminus \left\{ x=y=0\right\} \subset \mathbb{A}_{\ast }^{2}\times \mathrm{%
Spec}\left( \mathbb{C}\left[ u,v\right] \right)
\end{equation*}%
where $p\left( x,y\right) \in \mathbb{C}\left[ x,y\right] $ is a polynomial
divisible neither by $x$ nor by $y$ and satisfying $\deg _{x}p<m$ and $\deg
_{y}p<n$. It turns out that varieties $X\left( m,n,p\right) $ corresponding
to a polynomial $p$ of maximum possible degree $m+n-2$ form a distinguished
class. Namely, we have the following result:
\end{parn}
\begin{thm}
\label{thm:DeRhamArg} Let $X_{1}=X\left(m_{1},n_{1},p_{1}\right)$ and $%
X_{2}=X\left(m_{2},n_{2},p_{2}\right)$ be $\mathbb{A}^{1}$-bundles as above.
If $\deg p_{1}=m_{1}+n_{1}-2$ but $\deg p_{2}<m_{2}+n_{2}-2$ then $X_{1}$
and $X_{2}$ are not isomorphic as algebraic varieties.
\end{thm}

\begin{proof}
We will show that the cohomology class in $H_{dR}^{3}\left(X,\mathbb{C}%
\right)\simeq\mathbb{C}$ of an arbitrary nowhere vanishing algebraic $3$%
-form $\omega$ on $X=X\left(m,n,p\right)$ is trivial if $\deg p<m+n-2$ and a
generator otherwise. This prevents in particular the existence of an
isomorphism $f:X_{2}\overset{\sim}{\rightarrow}X_{1}$ : indeed, otherwise, similarly as in the
particular case above, the pull-back of a nowhere vanishing algebraic $3$-form $\omega$ on $X_{1}$
would be a nowhere vanishing algebraic $3$-form $f^{*}\omega$ on $X_{2}$
whose cohomology class $[f^{*}\omega]=f^{*}\left[\omega\right]\in
H_{dR}^{3}(X_{2},\mathbb{C})\simeq H_{dR}^{3}(X_{1},\mathbb{C})$ would
generate $H_{dR}^{3}(X_{2},\mathbb{C})$, a contradiction. Recall \ref%
{exa:LocalDesc} that $X=X\left(m,n,p\right)$ is covered by two principal
affine open subsets $X_{x}\simeq U_{x}\times\mathrm{Spec}\left(\mathbb{C}%
\left[t_{x}\right]\right)$ and $X_{y}\simeq U_{y}\times\mathrm{Spec}\left(%
\mathbb{C}\left[t_{y}\right]\right)$, where $t_{x}=x^{-m}u$ and $%
t_{y}=y^{-n}v$. Since every invertible function on $X$ is constant, a
nowhere vanishing algebraic $3$-form $\omega$ on $X$ is uniquely determined
locally by a pair of $3$-forms $\omega\mid_{X_{x}}=\lambda dx\wedge dy\wedge
dt_{x}\in\Omega_{U_{x}\times\mathbb{A}^{1}}^{3}$ and $\omega\mid_{X_{y}}=%
\lambda dx\wedge dy\wedge dt_{y}\in\Omega_{U_{y}\times\mathbb{A}^{1}}^{3}$,
where $\lambda\in\mathbb{C}^{*}$. Let $(\alpha_{x},\alpha_{y})=(\lambda
t_{x}dx\wedge dy,\lambda t_{y}dx\wedge dy)\in\Omega_{U_{x}\times\mathbb{A}%
^{1}}^{2}\times\Omega_{U_{y}\times\mathbb{A}^{1}}^{2}$ be local primitives
of $\omega\mid_{X_{x}}$ and $\omega\mid_{X_{y}}$ respectively. By definition
of the connecting homomorphism 
\begin{equation*}
\delta:H_{dR}^{2}(X_{x}\cap X_{y},\mathbb{C})\simeq H_{dR}^{2}(U_{x}\cap
U_{y},\mathbb{C})\overset{\sim}{\longrightarrow}H_{dR}^{3}\left(X,\mathbb{C}%
\right)
\end{equation*}
in the Mayer-Vietoris long exact sequence for the covering of $X$ by $X_{x}$
and $X_{y}$, the cohomology class of $\omega\in\Omega_{X}^{3}$ in $%
H_{dR}^{3}\left(X,\mathbb{C}\right)\simeq\Omega_{X}^{3}/d\Omega_{X}^{2}$
coincides with the image by $\delta$ of the cohomology class $\alpha\in
H_{dR}^{2}(X_{x}\cap X_{y},\mathbb{C})$ of the $2$-form 
\begin{eqnarray*}
(\alpha_{y}-\alpha_{x})\mid_{X_{x}\cap X_{y}} & = & (\lambda t_{y}dx\wedge
dy-\lambda t_{x}dx\wedge dy)\mid_{X_{x}\cap X_{y}} \\
& = & \lambda(t_{x}+x^{-m}y^{-n}p\left(x,y\right))dx\wedge dy-\lambda
t_{x}dx\wedge dy \\
& = & \lambda x^{-m}y^{-n}p\left(x,y\right)dx\wedge dy\in\Omega_{X_{x}\cap
X_{y}/\mathbb{C}}^{2}.
\end{eqnarray*}
Such a form is exact if and only if $x^{-m}y^{-n}p\left(x,y\right)$ does not
contain a term of the form $ax^{-1}y^{-1}$, where $a\in\mathbb{C}^{*}$, that
is, if and only if $\deg p<m+n-2$. Thus $\left[\omega\right]%
=\delta\left(\alpha\right)\in H_{dR}^{3}\left(X,\mathbb{C}\right)$ is either
trivial if $\deg p<m+n-2$ or a generator otherwise.
\end{proof}

\begin{cor}
\label{cor:Xmn_are_exotic} The total space of a nontrivial homogeneous $%
\mathbb{G}_{a}$-bundle $\rho:X\rightarrow\mathbb{A}_{*}^{2}$ of degree $%
-d<-2 $ is not isomorphic to $X_{1,1}\simeq\mathrm{SL}_{2}\left(\mathbb{C}%
\right)$.
\end{cor}

\begin{proof}
The variety $X_{1,1}=X\left(1,1,1\right)=\left\{ xv-yu=1\right\} $ belongs
to the class $X\left(m,n,p\right)$ with $\deg p=m+n-2$. On the other hand,
by virtue of Theorem \ref{thm:HomBunIsoType}, we may assume that $X\simeq
X\left(d-1,1,1\right)=\left\{ x^{d-1}v-yu=1\right\} $. Since $0=\deg p<d-2$
by hypothesis, the assertion follows from Theorem \ref{thm:DeRhamArg} above.
\end{proof}

\begin{example}
Theorem \ref{thm:DeRhamArg} implies that the total spaces of the $\mathbb{G}%
_{a}$-bundles 
\begin{equation*}
X=\left\{ x^{2}v-y^{2}u=1\right\} \rightarrow\mathbb{A}_{*}^{2}\quad\text{and%
}\quad X^{\prime }=\left\{ x^{2}v-y^{2}u=1+xy\right\} \rightarrow\mathbb{A}%
_{*}^{2}
\end{equation*}
are not isomorphic as algebraic varieties. However, $X$ and $X^{\prime }$
are biholomorphic $\mathbb{A}^{1}$-bundles over $\mathbb{A}_{*}^{2}$.
Indeed, the \v{C}ech $1$-cocycle $(xy)^{-2}\left(1+xy\right)\in
C^{1}(\{U_{x},U_{y}\},\mathcal{H}ol_{\mathbb{A}_{*}^{2}})$ defining $%
X^{\prime }$ is analytically cohomologous to the one $\left(xy\right)^{-2}%
\exp\left(xy\right)$, obtained by multiplying the \v{C}ech $1$-cocyle $%
\left(xy\right)^{-2}$ defining $X$ by the nowhere vanishing holomorphic
function $\exp\left(xy\right)$ on $\mathbb{A}_{*}^{2}$. In contrast with the
algebraic situation considered in the proof of Theorem \ref{thm:DeRhamArg},
the pull back by the corresponding biholomorphism $X\overset{\sim}{%
\rightarrow}X^{\prime }$ of the nowhere vanishing algebraic $3$-form $%
\omega=x^{-2}dx\wedge dy\wedge du\mid_{X^{\prime }}$ on $X^{\prime }$, whose
class generates $H_{dR}^{3}\left(X^{\prime },\mathbb{C}\right)$, is the
nowhere vanishing holomorphic $3$-form $x^{-2}\exp\left(-xy\right)dx\wedge
dy\wedge du\mid_{X}$. The later is analytically cohomologous to the
algebraic $3$-form $-x^{-1}ydx\wedge dy\wedge du\mid_{X}$, whose class
generates $H_{dR}^{3}\left(X,\mathbb{C}\right)$.
\end{example}

\begin{rem}
Since the cylinders $X\times\mathbb{A}^{1}$ and $X^{\prime }\times\mathbb{A}%
^{1}$ are algebraically isomorphic (see \ref{par:cancelRq}), $X$ and $%
X^{\prime }$ above provide a new example of biholomorphic complex algebraic
varieties for which algebraic cancellation fails. Of course, $X$ and $%
X^{\prime }$ are remote from affine spaces from a topological point of view.
But, in contrast with other families of $3$-dimensional counter-examples
constructed so far \cite{Dubouloz2011} \cite{Finston2008}, $X$ and $%
X^{\prime }$ have a trivial Makar-Limanov invariant (see \ref{rem:ML-triv}
above). It is interesting to relate the existing counter-examples to
Miyanishi's characterization of the affine $3$-space $\mathbb{A}^{3}$ \cite%
{Miyanishi1988}, which can be equivalently formulated as the fact that a smooth
affine threefold $X$ is algebraically isomorphic to $\mathbb{A}^{3}$ if and
only is satisfies the following conditions :

(i) There exists a regular function $f:X\rightarrow\mathbb{A}^{1}$ and a
Zariski open subset $U\subset\mathbb{A}^{1}$ such that $f^{-1}\left(U\right)%
\simeq U\times\mathbb{A}^{2}$,

(ii) All scheme theoretic fibers of $f$ are UFDs (i.e. $\Gamma\left(X_{c},%
\mathcal{O}_{X_{c}}\right)$ is a UFD for every fiber $X_{c}=f^{-1}\left(c%
\right)$, $c\in\mathbb{A}^{1}$),

(iii) $H^{3}\left(X,\mathbb{Z}\right)=0$.

\noindent The counter-examples obtained in \cite{Dubouloz2011} for
contractible affine threefolds satisfy (i) and (iii) but not (ii). On the
other hand, $X$ and $X^{\prime }$ above satisfy (i) and (ii) (by choosing
for instance the projection $\mathrm{pr}_{x}$ for $f$) but not (iii). The
Cancellation Problem for $\mathbb{A}^{3}$ itself is still open but we see
that cancellation fails whenever one of the necessary conditions (ii) or
(iii) above is relaxed.
\end{rem}

\section{Appendix}

\subsection{The Danilov-Gizatullin Isomorphism Theorem}

\indent\newline
\noindent This subsection is devoted to the proof of the following result,
which is a slight refinement of the so-called Danilov-Gizatullin Isomorphism
Theorem.

\begin{thm}
\label{thm:DanGizIso}Let $\nu:Y\rightarrow\mathbb{P}^{1}$ and $\nu^{\prime
}:Y^{\prime }\rightarrow\mathbb{P}^{1}$ be non trivial principal $\mathcal{O}%
_{\mathbb{P}^{1}}\left(-d\right)$-bundles for a certain $d\geq2$. Then there
exists an isomorphism $f:Y^{\prime }\overset{\sim}{\rightarrow}Y$ such that $%
f^{*}(\nu^{*}\mathcal{O}_{\mathbb{P}^{1}}\left(1\right))\simeq\left(\nu^{%
\prime }\right)^{*}\mathcal{O}_{\mathbb{P}^{1}}\left(1\right)$. In
particular, for a fixed $d\geq2$, the total spaces of non trivial $\mathcal{O%
}_{\mathbb{P}^{1}}\left(-d\right)$-bundles are all isomorphic as abstract
algebraic varieties.
\end{thm}

\begin{parn}
The Danilov-Gizatullin Theorem \cite[Theorem 5.8.1 ]{Gizatullin1977} is
actually stated as the fact that the isomorphy type of the complement of an
ample section $C$ in a Hirzebruch surface $\pi_{n}:\mathbb{F}_{n}=\mathbb{P}%
\left(\mathcal{O}_{\mathbb{P}^{1}}\oplus\mathcal{O}_{\mathbb{P}%
^{1}}\left(-n\right)\right)\rightarrow\mathbb{P}^{1}$, $n\geq0$, (see e.g. 
\cite[V.2]{Hartshorne1977}) depends only on the self-intersection $C^{2}$ of 
$C$, whence, in particular, depends neither on the ambient surface nor on
the choice of a particular section. The relation with Theorem \ref%
{thm:DanGizIso} above is given by the observation that a non trivial $%
\mathcal{O}_{\mathbb{P}^{1}}\left(-d\right)$-bundle $\nu:Y\rightarrow\mathbb{%
P}^{1}$ always arises as the complement of an ample section $C$ with
self-intersection $C^{2}=d$ in a suitable Hirzebruch surface \cite[Remark
4.8.6]{Gizatullin1977}. Indeed, letting $0\rightarrow\mathcal{O}_{\mathbb{P}%
^{1}}\rightarrow\mathcal{E}\rightarrow\mathcal{O}_{\mathbb{P}%
^{1}}\left(d\right)\rightarrow0$ be an extension of locally free sheaves on $%
\mathbb{P}^{1}$ representing the isomorphy class of $Y$ in $H^{1}(\mathbb{P}%
^{1},\mathcal{O}_{\mathbb{P}^{1}}\left(-d\right))\simeq\mathrm{Ext}_{\mathbb{%
P}^{1}}^{1}\left(\mathcal{O}_{\mathbb{P}^{1}}\left(d\right),\mathcal{O}_{%
\mathbb{P}^{1}}\right)$, $Y$ is isomorphic to the complement in the $\mathbb{%
P}^{1}$-bundle $\pi:S=\mathrm{Proj}_{\mathbb{P}^{1}}\left(\mathrm{Sym}\left(%
\mathcal{E}\right)\right)\rightarrow\mathbb{P}^{1}$ of the section $C$
determined by the surjection $\mathcal{E}\rightarrow\mathcal{O}_{\mathbb{P}%
^{1}}\left(d\right)$. Since $\mathcal{E}$ is a decomposable locally free
sheaf of rank $2$, degree $-d$, equipped with a surjection onto $\mathcal{O}%
_{\mathbb{P}^{1}}\left(d\right)$, it is isomorphic $\mathcal{O}_{\mathbb{P}%
^{1}}\left(a\right)\oplus\mathcal{O}_{\mathbb{P}^{1}}\left(d-a\right)$ for a
certain $a\in\mathbb{Z}$ such that, up to replacing $a$ by $d-a$, we have
either $a=0$ or $d-a\geq a>0$. Therefore, $S\simeq\mathrm{Proj}_{\mathbb{P}%
^{1}}\left(\mathrm{Sym}\left(\mathcal{E}\otimes\mathcal{O}_{\mathbb{P}%
^{1}}\left(a-d\right)\right)\right)\simeq\mathbb{F}_{n}$, where $n=d-2a\geq0$%
, with the section $C$ determined by a surjection $\mathcal{O}_{\mathbb{P}%
^{1}}\oplus\mathcal{O}_{\mathbb{P}^{1}}\left(-n\right)\rightarrow\mathcal{O}%
_{\mathbb{P}^{1}}\left(a\right)$. Letting $C_{0}$ be a section with
self-intersection $-n\leq0$ and $\ell$ be a fiber of $\pi_{n}$, we have $%
C\sim C_{0}+\left(d-a\right)\ell$, which implies that $C^{2}=d$.
Furthermore, $a=0$ if and only if the above extension splits, that is, if
and only if $\nu:Y\rightarrow\mathbb{P}^{1}$ is the trivial $\mathcal{O}_{%
\mathbb{P}^{1}}\left(-d\right)$-bundle. Otherwise, $d-a>n$, and so, $C$ is
the support of an ample divisor on $S$ \cite[2.20 p. 382 ]{Hartshorne1977}.
\end{parn}

\begin{parn}
\label{par:Proof-Strat} The existence of an isomorphism $f$ with the
required property can actually be derived from a careful reading of the
recent proof of the Danilov-Gizatullin Theorem given in \cite{Flenner2009}.
However, for the convenience of the reader, we provide a complete argument.
Our strategy is very similar to the one in \emph{loc. cit.}: we establish
that the total space of a non trivial $\mathcal{O}_{\mathbb{P}%
^{1}}\left(-d\right)$-bundle $\nu:Y\rightarrow\mathbb{P}^{1}$ is equipped
with a certain type of smooth fibration $\theta:Y\rightarrow\mathbb{A}^{1}$
with general fibers isomorphic to $\mathbb{A}^{1}$ which, for a fixed $%
d\geq2 $, admits a unique model $\theta_{d}:S\left(d\right)\rightarrow%
\mathbb{A}^{1} $ up to isomorphism of fibrations. Since $\mathrm{Pic}%
\left(Y\right)\simeq\mathrm{CaCl}\left(Y\right)$ is generated by the class
of a fiber of $\nu$, Theorem \ref{thm:DanGizIso} will then follows from the
additional observation that one can always choose a special isomorphism of
fibrations $\psi:Y\overset{\sim}{\rightarrow}S\left(d\right)$ which maps a
suitable fiber of $\nu$ onto a fixed irreducible component $\Delta$ of a
fiber of $\theta_{d}$.
\end{parn}

\begin{parn}
The fibration $\theta:Y\rightarrow\mathbb{A}^{1}$ is constructed as follows.
We may suppose that the non trivial $\mathcal{O}_{\mathbb{P}%
^{1}}\left(-d\right)$-bundle $\nu:Y\rightarrow\mathbb{P}^{1}$ is embedded in
a Hirzebruch surface $\pi_{n}:\mathbb{F}_{n}\rightarrow\mathbb{P}^{1}$ for a
certain $n\geq0$ as the complement of an ample section $C$ with $%
C^{2}=d\geq2 $. Now suppose that there exists another section $\tilde{C}$ of 
$\pi_{n}$ intersecting $C$ in a unique point $q$, with multiplicity $C\cdot%
\tilde{C}=d-1$. Letting $\ell=\pi_{n}^{-1}\left(\pi_{n}\left(q\right)\right)$%
, the divisors $C$ and $\tilde{C}+\ell$ are linearly equivalent and define a
pencil of rational curves $g:\mathbb{F}_{n}\dashrightarrow\mathbb{P}^{1}$
with $q$ as a unique proper base point. This pencil restricts on $Y=\mathbb{F%
}_{n}\setminus C$ to a smooth surjective morphism $\theta:Y\rightarrow B=%
\mathbb{P}^{1}\setminus\left\{ \overline{g}\left(C\right)\right\} \simeq%
\mathbb{A}^{1}$ with general fibers isomorphic to $\mathbb{A}^{1}$ and with
a unique degenerate fiber, say $\theta^{-1}\left(0\right)$ up to the choice
of a suitable coordinate $x$ on $B\simeq\mathbb{A}^{1}$, consisting of the
disjoint union of $\tilde{C}_{0}=\tilde{C}\setminus\left\{ q\right\} \simeq%
\mathbb{A}^{1}$ and $\ell_{0}=\ell\setminus\left\{ q\right\} \simeq\mathbb{A}%
^{1}$. A minimal resolution $\overline{g}:W\rightarrow\mathbb{P}^{1}$ of $g:%
\mathbb{F}_{n}\dashrightarrow\mathbb{P}^{1}$ is obtained from $\mathbb{F}%
_{n} $ by blowing up $d$ times the point $q$, with successive exceptional
divisors $E_{1},\ldots,E_{d}$, the last exceptional divisor $E_{d}$ being a
section of $\overline{g}$. The proper transform of $C$ in $W$ is a full
fiber of $\overline{g}$, whereas the proper transforms of $\tilde{C}$ and $%
\ell$ are both $-1$-curves contained in the unique degenerate fiber $%
\overline{g}^{-1}\left(0\right)=E_{1}+\cdots+E_{d-1}+\tilde{C}+\ell$ of $%
\overline{g}$. Since $E_{1}\cup\dots\cup E_{d-1}$ is a chain of $%
\left(-2\right)$-curves, by contracting successively $\ell$, $%
E_{1},\ldots,E_{d-2}$ and $\tilde{C}$ we obtain a birational morphism $%
\overline{\tau}:W\rightarrow\mathbb{F}_{1}$. The later restricts to a
morphism $\tau:Y\simeq W\setminus C\cup\bigcup_{i=1}^{d}E_{i}\rightarrow%
\mathbb{F}_{1}\setminus C\cup E_{d}\simeq B\times\mathbb{A}^{1}$ of schemes
over $B$, inducing an isomorphism $Y\setminus\tilde{C}_{0}\cup\ell_{0}%
\overset{\sim}{\rightarrow}B\setminus\left\{ 0\right\} \times\mathbb{A}^{1}$
and contracting $\tilde{C}_{0}$ and $\ell_{0}$ to distinct points supported
on $\left\{ 0\right\} \times\mathbb{A}^{1}\subset B\times\mathbb{A}^{1}$
(see Figure \ref{fig:figBir} below).
\end{parn}

\begin{figure}[!htb]
\psset{linewidth=0.5pt} 
\begin{pspicture}(-1.8,4.5)(10,-2)

\rput(-3,0){
 \pscustom{
\pscurve(1, -0.3)(2.25,-0.4)(3.5,-0.3)
\psline(3.5,0)(3.5,1.7)
\pscurve (2.25,1.6)(1,1.7)
\psline(1,1.7)(1,-0.3)}

\psline[linestyle=dotted](1.5,-0.3)(1.5,1.65)
\psline[linestyle=dotted](2.5,-0.35)(2.5,1.55)
\psline[linestyle=dotted](3,-0.3)(3,1.6)

\rput(4,1.8){$\mathbb{F}_n$}

\rput(0.8,1.3){{\small $\tilde{C}$}}
\rput(0,1.3){\pscurve[linewidth=1pt](1,-0.2)(1.5,-0.5)(2,-0.7)(2.7,-0.5)(3,-0.3)(3.5,0.1)}

\rput(2.1,0.45){{\scriptsize $q$}}
\rput(0,1.3){\pscurve[linewidth=0.2pt,linecolor=gray,linestyle=dashed](1,-1.1)(1.5,-0.9)(2,-0.7)(2.7,-1)(3,-1.1)(3.5,-1)}
\rput(0,1.3){\pscurve[linewidth=0.2pt,linecolor=gray,linestyle=dashed](1,-0.65)(1.5,-0.7)(2,-0.7)(2.7,-0.75)(3,-0.85)(3.5,-0.75)}
\rput(0,1.3){\pscurve[linewidth=0.2pt,linecolor=gray,linestyle=dashed](1,-0.5)(1.5,-0.6)(2,-0.7)(2.7,-0.7)(3,-0.75)(3.5,-0.6)}

\rput(0.8,0.4){{\small $C$}}
\rput(0,1.3){\pscurve[linewidth=1pt](1,-0.9)(1.5,-0.8)(2,-0.7)(2.7,-0.9)(3,-1)(3.5,-0.9)}

\rput(2.2,0){{\small $\ell$}}
\psline[linewidth=1pt](2,-0.4)(2,1.6)

\rput(4,-1.2){$\mathbb{P}^1$}
\rput(0,-1){ \pscurve(1, -0.3)(2.25,-0.4)(3.5,-0.3)}
\psline{->}(2.25,-0.5)(2.25,-1.3)\rput(2.6,-0.9){{\small $\pi_n$}}

}

\rput(2.5,3){

\rput(3.3,1.3){$W$}
\psline(1,1)(3,1)
\rput(1.8,1.2){{\scriptsize $E_d$}}
\rput(2.3,1.2){{\scriptsize -$1$}}
\psline[linewidth=1pt](3,1)(3,-1)
\rput(3.2,0){{\small $C$}}
\rput(3.2,-0.4){{\scriptsize $0$}}
\psline(1,1.1)(1.4,0.2)
\rput(0.7,0.8){{\scriptsize $E_{d-1}$}}
\rput(1.3,0.8){{\scriptsize -$2$}}
\psline(1.4,0.4)(0.7,-0.6)
\rput(0.6,0.1){{\scriptsize $E_{d-2}$}}
\rput(1.3,-0.1){{\scriptsize -$2$}}
\psline[linestyle=dashed](0.7,-0.3)(1.2,-1.2)
\psline(1.2,-1)(0.8,-1.9)
\rput(0.8,-1.2){{\scriptsize $E_1$}}
\rput(1.1,-1.7){{\scriptsize -$2$}}
\psline[linewidth=1pt](1,0.5)(1.9,0.5)
\rput(1.7,0.3){{\small $\tilde{C}$}}
\rput(1.7,0.65){{\scriptsize -$1$}}
\psline[linewidth=1pt](0.8,-1.5)(1.7,-1.5)
\rput(1.6,-1.7){{\small $\ell$}}
\rput(1.6,-1.35){{\scriptsize -$1$}}

\psline[linewidth=0.2pt,linecolor=gray,linestyle=dashed](2,1)(2,-1)
\psline[linewidth=0.2pt,linecolor=gray,linestyle=dashed](2.25,1)(2.25,-1)
\psline[linewidth=0.2pt,linecolor=gray,linestyle=dashed](2.5,1)(2.5,-1)

\psline(1,-2.5)(3,-2.5)
\rput(3.3,-2.5){$\mathbb{P}^1$}

\psline{->}(2,-1.8)(2,-2.3)\rput(2.3,-2){{\small $\overline{g}$}}

}

\rput(7,0.5){

\rput(3.3,1.3){$\mathbb{F}_1$}
\psline(1,1)(3,1)
\rput(1.8,1.2){{\scriptsize $E_d$}}
\rput(2.3,1.2){{\scriptsize -$1$}}
\psline[linewidth=1pt](3,1)(3,-1)
\rput(3.2,0){{\small $C$}}
\rput(3.2,-0.3){{\scriptsize $0$}}
\psline(1,1)(1,-1)
\rput(1.4,0){{\scriptsize $E_{d-1}$}}
\rput(1.2,-0.3){{\scriptsize $0$}}
\rput(1,0.4){\textbullet}
\rput(0.6,0.4){{\scriptsize $\overline{\tau}(\tilde{C})$}}
\rput(1,-0.5){\textbullet}
\rput(0.6,-0.5){{\scriptsize $\overline{\tau}(\ell)$}}
\psline[linewidth=0.2pt,linecolor=gray,linestyle=dashed](2,1)(2,-1)
\psline[linewidth=0.2pt,linecolor=gray,linestyle=dashed](2.25,1)(2.25,-1)
\psline[linewidth=0.2pt,linecolor=gray,linestyle=dashed](2.5,1)(2.5,-1)

\psline(1,-2)(3,-2)
\rput(3.3,-2){$\mathbb{P}^1$}

\psline{->}(2,-1.2)(2,-1.9)\rput(2.3,-1.5){{\small $\pi_1$}}

}

\psline{->}(2.5,2.5)(1.5,1.5)
\psline{->}(6.5,2.5)(7.5,1.5)
\rput(7.2,2.2){{\small $\overline{\tau}$}}
\end{pspicture}
\caption{Resolution of the pencil $g:\mathbb{F}_n\dashrightarrow \mathbb{P}%
^1 $ and the contraction $\overline{\protect\tau} :W\rightarrow \mathbb{F}_1$%
.}
\label{fig:figBir}
\end{figure}
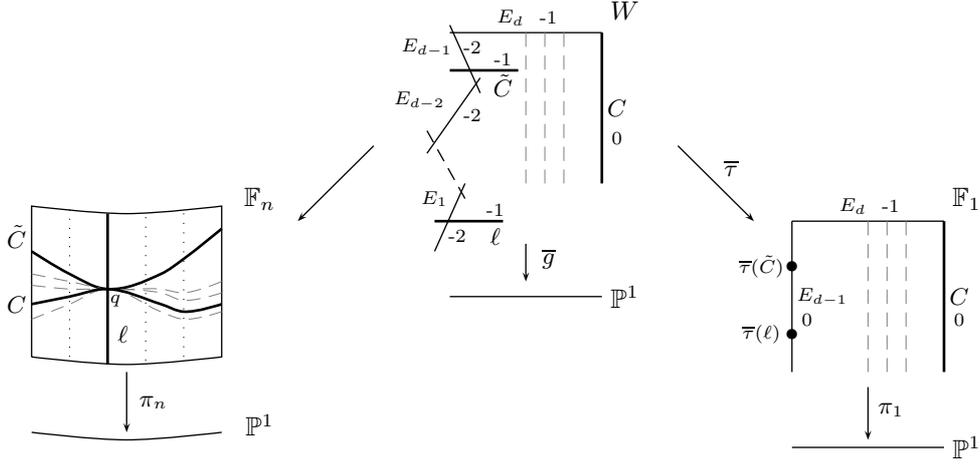

\begin{parn}
Up to a suitable choice of coordinate on the second factor of $B\times%
\mathbb{A}^{1}$, we may assume that $\tau(\tilde{C}_{0})=\left(0,1\right)$
and $\tau(\ell_{0})=\left(0,0\right)$. It then follows from the construction
of $\tau$ that there exists isomorphisms $Y\setminus\ell_{0}\simeq B\times%
\mathrm{Spec}\left(\mathbb{C}\left[u_{1}\right]\right)$ and $Y\setminus%
\tilde{C}_{0}\simeq B\times\mathrm{Spec}\left(\mathbb{C}\left[u_{2}\right]%
\right)$ for which the restrictions of $\tau:Y\rightarrow B\times\mathbb{A}%
^{1}$ to $Y\setminus\ell_{0}$ and $Y\setminus\tilde{C}_{0}$ coincide
respectively with the birational morphisms 
\begin{equation*}
Y\setminus\ell_{0}\rightarrow B\times\mathbb{A}^{1},\;\left(x,u_{1}\right)%
\mapsto\left(x,xu_{1}+1\right)\quad\text{and}\quad Y\setminus\tilde{C}%
_{0}\rightarrow B\times\mathbb{A}^{1},\;\left(x,u_{2}\right)%
\mapsto(x,x^{d-1}u_{2}+\sum_{i=1}^{d-2}a_{i}x^{i}),
\end{equation*}
where the complex numbers $a_{1},\ldots,a_{d-1}$ depend on the successive
centers of $\overline{\tau}:W\rightarrow\mathbb{F}_{1}$. Replacing $u_{1}$
by $v_{1}=u_{1}-\sum_{i=1}^{d-2}a_{i}x^{i-1}$ yields a new isomorphism $%
Y\setminus\ell_{0}\simeq B\times\mathrm{Spec}\left(\mathbb{C}\left[v_{1}%
\right]\right)$. Letting $v_{2}=u_{2}$, we can eventually identify $%
\theta:Y\rightarrow B$ with the surface $\theta_{d}:S\left(d\right)%
\rightarrow B$ obtained by gluing two copies $S_{1}=B\times\mathrm{Spec}%
\left(\mathbb{C}\left[v_{1}\right]\right)$ and $S_{2}=B\times\mathrm{Spec}%
\left(\mathbb{C}\left[v_{2}\right]\right)$ of $B\times\mathbb{A}^{1}$ along $%
\left(B\setminus\left\{ 0\right\} \right)\times\mathbb{A}^{1}$ by the
isomorphism 
\begin{equation*}
S_{1}\supset\left(B\setminus\left\{ 0\right\} \right)\times\mathbb{A}%
^{1}\ni\left(x,v_{1}\right)\mapsto\left(x,x^{2-d}v_{1}+x^{1-d}\right)\in%
\left(B\setminus\left\{ 0\right\} \right)\times\mathbb{A}^{1}\subset S_{2}.
\end{equation*}
\end{parn}

\begin{parn}
\label{par:Section} Summing up, starting from a section $\tilde{C}$ of $%
\pi_{n}$ intersecting $C$ in a unique point $q$ with multiplicity $d-1$, we
constructed an isomorphism $Y=\mathbb{F}_{n}\setminus C\overset{\sim}{%
\rightarrow}S\left(d\right)$ which maps $\nu^{-1}(\pi_{n}(q))=\ell_{0}$
isomorphically onto the curve $\Delta=\left\{ x=0\right\} \subset S_{2}$. So
Theorem \ref{thm:DanGizIso} eventually follows from the next lemma (see also 
\cite[Prop. 4.8.11]{Gizatullin1977}), which guarantees the existence of
sections $\tilde{C}$ with the required property.
\end{parn}

\begin{lem}
Let $\pi_{n}:\mathbb{F}_{n}\rightarrow\mathbb{P}^{1}$, $n\geq0$, be a
Hirzebruch surface and let $C\subset\mathbb{F}_{n}$ be an ample section with
self-intersection $d\geq2$. Then given a general point $q\in C$, there
exists a section $\tilde{C}$ such that $C\cdot\tilde{C}=\left(d-1\right)q$.
\end{lem}

\begin{proof}
The existence of a section $\tilde{C}$ such that $C\cdot\tilde{C}%
=\left(d-1\right)q$ for a certain $q\in C$ is equivalent to the existence of
a rational section of the induced $\mathbb{A}^{1}$-bundle $%
\nu=\pi_{n}\mid_{Y}:Y=\mathbb{F}_{n}\setminus C\rightarrow\mathbb{P}^{1}$
with a pole of order $d-1$ at the point $\pi_{n}\left(q\right)$. Since $%
\nu:Y\rightarrow\mathbb{P}^{1}=\mathrm{Proj}\left(\mathbb{C}\left[w_{0},w_{1}%
\right]\right)$ is a non trivial $\mathcal{O}_{\mathbb{P}^{1}}\left(-d%
\right) $-bundle, we can find local trivializations $\tau_{1}:\nu^{-1}%
\left(U_{w_{1}}\right)\overset{\sim}{\rightarrow}U_{w_{1}}\times\mathrm{Spec}%
\left(\mathbb{C}\left[u\right]\right)$ and $\tau_{0}:\nu^{-1}\left(U_{w_{0}}%
\right)\overset{\sim}{\rightarrow}U_{w_{0}}\times\mathrm{Spec}\left(\mathbb{C%
}\left[v\right]\right)$ such that, letting $z=w_{0}/w_{1}$, the isomorphism $%
\tau_{0}\circ\tau_{1}^{-1}\mid_{U_{w_{1}}\cap U_{w_{0}}}$ has the form $%
\left(z,v\right)\mapsto\left(z^{-1},z^{d}u+p\left(z\right)\right)$ for a
nonzero polynomial $p\left(z\right)\in z\mathbb{C}\left[z\right]$ of degree $%
\deg p<d$. In these trivializations, a rational section of $\nu$ with pole
of order $d-1$ at a point $\lambda=\pi_{n}\left(q\right)\in U_{w_{1}}\cap
U_{w_{0}}$ is uniquely determined by a rational function $%
f_{1}:U_{w_{1}}\dashrightarrow\mathbb{A}^{1}$, $z\mapsto\left(z-\lambda%
\right)^{1-d}s\left(z\right)$ such that $\lambda\neq0$, $s\left(z\right)\in%
\mathbb{C}\left[z\right]$ does not vanish at $\lambda$, and such that $%
z^{d}s\left(z\right)+\left(z-\lambda\right)^{d-1}p\left(z\right)\in
O_{\infty}\left(z^{d-1}\right)$. Indeed, the last condition guarantees that $%
z^{d}f_{1}+p\left(z\right)$ extends to a rational function $%
f_{0}:U_{w_{0}}\dashrightarrow\mathbb{A}^{1}$ regular at the origin, whence
that the local rational sections $f_{1}$ and $f_{0}$ of $\nu$ glue to a
global one $\sigma:\mathbb{P}^{1}\dashrightarrow Y$ with a unique pole at $%
\lambda\in U_{w_{1}}\cap U_{w_{0}}$, of order $d-1$. Writing $%
(z-\lambda)^{d-1}p\left(z\right)=\alpha_{\lambda}\left(z\right)+z^{d}\beta_{%
\lambda}\left(z\right)$ where $\alpha_{\lambda}\left(z\right)\in\mathbb{C}%
\left[z\right]$ is a non zero polynomial of degree $\deg\alpha_{\lambda}\leq
d-1$, we have necessarily $s\left(z\right)=-\beta_{\lambda}\left(z\right)$,
which forces in turn $s\left(\lambda\right)=\lambda^{-d}\alpha_{\lambda}%
\left(\lambda\right)$. Letting $p\left(z\right)=a_{1}z+\cdots+a_{d-1}z^{d-1}$%
, a direct computation shows that 
\begin{equation*}
s\left(\lambda\right)=\sum_{i=0}^{d-2}\left(-1\right)^{i+1}\tbinom{d-2}{i}%
a_{i+1}\lambda^{i}.
\end{equation*}
Since $r\left(z\right)=\sum_{i=0}^{d-2}\left(-1\right)^{i+1}\tbinom{d-2}{i}%
a_{i+1}z^{i}$ is a nonzero polynomial, it follows that for every $%
\lambda\in(U_{w_{1}}\cap U_{w_{0}})\setminus V\left(r\left(z\right)\right)$
there exists a (unique) rational section $\sigma:\mathbb{P}%
^{1}\dashrightarrow Y$ with pole of order $d-1$ at $\lambda$.
\end{proof}

\subsection{A topological characterization of the affine $2$-sphere}

\indent\newline
\noindent The fact that a smooth affine surface $X$ diffeomorphic to $%
\mathbb{S}_{\mathbb{C}}^{2}=\left\{ z_{1}^{2}+z_{2}^{2}+z_{3}^{2}=1\right\}
\subset \mathbb{A}_{\mathbb{C}}^{3}$ is algebraically isomorphic to $\mathbb{%
S}_{\mathbb{C}}^{2}$ is probably folklore. We provide a proof because of the
lack of an appropriate reference.

\begin{parn}
Let $\mathbb{F}_{0}=\mathbb{P}^{1}\times \mathbb{P}^{1}$ with bi-homogeneous
coordinates $(\left[ u_{0}:u_{1}\right] ,\left[ v_{0}:v_{1}\right] )$. Via
the open immersion 
\begin{equation*}
j:\mathbb{S}_{\mathbb{C}}^{2}\rightarrow \mathbb{F}_{0},\;\left(
z_{1},z_{2},z_{3}\right) \mapsto \left( \left[ z_{1}+iz_{2}:z_{3}+1\right] ,%
\left[ z_{1}+iz_{2}:1-z_{3}\right] \right) 
\end{equation*}%
we may identify $\mathbb{S}_{\mathbb{C}}^{2}$ with the complement in $%
\mathbb{F}_{0}$ of the diagonal $\Delta =\left\{
u_{0}v_{1}+u_{1}v_{0}=0\right\} $. The restriction to $\mathbb{S}_{\mathbb{C}%
}^{2}$ of the first projection $\mathbb{F}_{0}\rightarrow \mathbb{P}^{1}$ is
a locally trivial $\mathbb{A}^{1}$-bundle, whence a trivial $\mathbb{R}^{2}$%
-bundle in the euclidean topology. Thus 
\begin{equation*}
H_{i}\left( \mathbb{S}_{\mathbb{C}}^{2},\mathbb{Z}\right) =%
\begin{cases}
\mathbb{Z} & \text{if }i=0,2 \\ 
0 & \text{otherwise}.%
\end{cases}%
\end{equation*}%
Furthermore, since $\Delta ^{2}=2$, the \emph{fundamental group at infinity} 
$\pi _{1}^{\infty }\left( \mathbb{S}_{\mathbb{C}}^{2}\right) $ of $\mathbb{S}%
_{\mathbb{C}}^{2}$ (see e.g. \cite{Gurjar1984,Ramanujam1971}) is isomorphic
to $\mathbb{Z}_{2}$. It turns out that these topological invariants provide
a characterization of $\mathbb{S}_{\mathbb{C}}^{2}$ among all smooth affine
surfaces, namely:
\end{parn}

\begin{thm}
\label{thm:NoExotic2-sphere} A smooth affine surface $X$ with the homology
type and the homotopy type at infinity of $\mathbb{S}_{\mathbb{C}}^{2}$ is
isomorphic to $\mathbb{S}_{\mathbb{C}}^{2}$.
\end{thm}

\begin{proof}
The finiteness of $\pi _{1}^{\infty }\left( X\right) $ implies that $X$ has 
logarithmic Kodaira dimension $\bar{\kappa}(X)=-\infty $ \cite{Gurjar1988}
whence that $X$ is affine ruled. It follows that $X$ admits a completion into a smooth, projective, 
birationally ruled surface $p:V\rightarrow C$, where $C$ is a smooth projective curve. 
One can further assume that the boundary $D:=V\setminus X$ is a connected divisor
with simple normal crossings that can be written as $D=B\cup $\ $G_{0}\cup G_{1}\cup \cdots \cup G_{s},$ where $B$
is a section of $p$ and the $G_{i}$ are disjoint trees of smooth rational
curves contained in the fibers of $p$ \cite[I.2]{Miyanishi1981}. 
In particular, the dual graph of $D$ is a tree. 

The hypotheses imply that $H_{i}\left( X,\mathbb{Z}\right) =0$ for $i=1,3,4$
and so $H^{i}\left( X,\mathbb{Z}\right) =0$ for $i=1,3,4$, whereas $%
H^{2}\left( X,\mathbb{Z}\right) \simeq H_{2}\left( X,\mathbb{Z}\right)
\simeq \mathbb{Z}$ by the universal coefficient Theorem. By Poincar\'e
-Lefschetz duality, we have $H^{i}\left( \left( V,D\right) ,\mathbb{Z}%
\right) \simeq H_{4-i}\left( X,\mathbb{Z}\right) $ and $H_{i}\left( \left(
V,D\right) ,\mathbb{Z}\right) \simeq H^{4-i}\left( X,\mathbb{Z}\right) $,
and so, these groups are zero for $i=0,1$ and $3$, and isomorphic to $%
\mathbb{Z}$ for $i=2,4$. From the long exact sequences of (co)homology of
pairs 
\begin{eqnarray*}
& \cdots \rightarrow H_{\ast }\left( D\right) \stackrel{\partial _{\ast }}{\rightarrow}  
H_{\ast }\left( V\right) \rightarrow H_{\ast }\left( V,D\right) 
\rightarrow  H_{\ast -1}\left( D\right) \rightarrow  \cdots &   \\
& \cdots  \rightarrow  H^{\ast -1}\left( D\right)  \stackrel{\partial^{\ast }}{
\rightarrow }  H^{\ast }\left( V,D\right)  \rightarrow  H^{\ast }\left( V\right)
 \rightarrow  H^{\ast }\left( D\right)  \rightarrow  \cdots & 
\end{eqnarray*}
we get $H_{3}\left( D,\mathbb{Z}\right) \simeq H_{3}\left( V,\mathbb{Z}%
\right) \simeq 0$ and so $H^{1}\left( V,\mathbb{Z}\right) =0$ by Poincar\'e
duality. Similarly, $H^{3}\left( D,\mathbb{Z}\right) \simeq H^{3}\left( V,%
\mathbb{Z}\right) =0$ and $H_{1}\left( V,\mathbb{Z}\right)=0$. It follows that 
$H^{1}\left( D,\mathbb{Z}\right)$ and $H_{1}\left( D,\mathbb{Z}\right)$ are either 
simultaneously $0$ or isomorphic to $\mathbb{Z}$. In the latter case $D$ would contain a cycle of
rational curves, which is impossible from the above description of $D$. Thus 
$H_{1}\left( D,\mathbb{Z}\right) =0$ and so $D$ is a tree of nonsingular rational curves. This
implies in turn that $H^{1}\left( V,\mathcal{O}_{V}\right) =\left\{
0\right\} $ for otherwise $D$ would be contained in a fiber of the Albanese
morphism $q:V\rightarrow {\rm Alb}(V)$, in contradiction with the fact that $D$ is the support
of an ample divisor as $X$ is affine. 
Since $\pi _{1}^{\infty }\left( X\right) \simeq \pi
_{1}\left( \mathbb{S}_{\mathbb{C}}^{2}\right) \simeq \mathbb{Z}_{2}$ by
hypothesis, Theorem 1 in \cite{Gurjar1984} and its proof imply that $D$ is a
chain. Therefore, up to replacing $V$ by another minimal completion of $X$
obtained from $V$ by a sequence of blow-ups and blow-downs with centers
outside $X$, we may assume that $D=D_{0}\cup D_{1}\cup \cdots \cup D_{s}$,
where $D_{i}\cdot D_{j}=1$ if $\left\vert i-j\right\vert =1$ and $0$
otherwise, $D_{0}^{2}=D_{1}^{2}=0$ and $D_{i}^{2}\leq -2$ for every $%
i=2,\ldots ,s$ (see e.g. \cite[Lemma 2.7 and 2.9]{Dubouloz2004}). With this
description, one checks easily that $\pi _{1}^{\infty }\left( X\right)
\simeq \mathbb{Z}_{2}$ if and only if $s=2$ and $D_{2}^{2}=-2$. By
blowing-up the point $D_{0}\cap D_{1}$ and contracting the proper transforms
of $D_{0}$, $D_{1}$ and $D_{2}$, we reach a completion $V_{0}$ of $X$ by a
smooth rational curve $B$ with self-intersection $2$. It follows from
Danilov-Gizatullin's classification \cite{Gizatullin1977} that $V_{0}\simeq 
\mathbb{F}_{0}$ and that $B$ is of type $\left( 1,1\right) $. Since the
automorphism group of $\mathbb{F}_{0}$ acts transitively on the set of
smooth curves of type $\left( 1,1\right) $, we finally obtain that $X\simeq 
\mathbb{F}_{0}\setminus \Delta \simeq \mathbb{S}_{\mathbb{C}}^{2}$.

\end{proof}

\bibliographystyle{amsplain}

\end{document}